\documentclass[a4paper, 12pt, dvipdfmx]{amsart}
 \usepackage{amssymb, amsmath, amsthm}
 \usepackage{amscd}
 \usepackage{color}
 \usepackage{graphicx}
 
\title
[Self and partial gluing]
{Self and partial gluing theorems for Alexandrov spaces with a lower curvature bound} 

\author{Ayato Mitsuishi}
\email[A.~Mitsuishi]{mitsuishi@math.gakushuin.ac.jp}
\date{\today}

\theoremstyle{plain}
\newtheorem{theorem}{Theorem}[section]
\newtheorem{lemma}[theorem]{Lemma}
\newtheorem{corollary}[theorem]{Corollary}
\newtheorem{proposition}[theorem]{Proposition}
\newtheorem{definition}[theorem]{Definition}
\newtheorem{remark}[theorem]{Remark}
\newtheorem{example}[theorem]{Example}

\newcommand{\gexp}[0]{\mathrm{gexp}}
\newcommand{\snk}[0]{\mathrm{sn}_\kappa}
\newcommand{\csk}[0]{\mathrm{cs}_\kappa}

\makeatletter

\@addtoreset{equation}{section}
\makeatother

\begin{document}
\maketitle

\begin{abstract}
This paper is devoted to prove that 
if an Alexandrov space of curvature not less than $\kappa$ with a codimension one extremal subset which admits an isometric involution with respect to the induced length metric, then the metric space obtained by gluing the extremal subset along the isometry is an Alexandrov space of curvature not less than $\kappa$. 
This is a generalization of Perelman's doubling and Petrunin's gluing theorems.
\end{abstract}

\section{Introduction and results}
Alexandrov spaces (of curvature bounded from below) are generalized objects of complete Riemannian manifolds. 
Such spaces naturally appear as the Gromov-Hausdorff limits of complete Riemannian manifolds 
under a uniform lower sectional curvature bound. 
The quotients of complete Riemannian manifolds by isometric actions possibly with fixed points are also Alexandrov spaces. 
Alexandrov spaces are stable in several geometric constructions, 
which are, taking the product of two spaces and taking the cone and the join of spaces of curvature not less than one. 
On the other hands, metric spaces having an upper curvature bound in the sense of Alexandrov, which are called CAT-spaces, 
are also stable in the geometric constructions mentioned above (product, cone, join). 
Reshetnyak (\cite{Re}) proved that for two (or more many) CAT-spaces and their convex subsets, if the convex subsets are isometric to each other, then the metric space obtained by gluing the CAT-spaces along the convex subsets via the isometries is again a CAT-space. 
For Alexandrov spaces, such a construction does not work in general. 
For instance, a set of single point in the plane is convex and the metric space obtained by gluing two planes at the base points admits branching geodesics, that can not admit Alexandrov's lower curvature bound. 

On the other hands, Perelman (\cite{Per:Alex2}) and Petrunin (\cite{Pet:appl}) proved that Alexandrov's lower curvature bounds are stable by gluing thier boundaries via isometric mapping with respect to the induced length metric. 
Exactly, Petrunin proved 
\begin{theorem}[\cite{Pet:appl}] \label{thm:Pet gluing}
Let $M_1$ and $M_2$ be Alexandrov spaces of the same dimension and having the same lower curvature bound $\kappa$ with non-empty boundaries $\partial M_1$ and $\partial M_2$. 
If there is an isometry $f : \partial M_1 \to \partial M_2$ with respect to the induced length metric, 
then the metric space $M_1 \cup_f M_2$ obtained by gluing $M_1$ and $M_2$ along thier boundaries via $f$ is an Alexandrov space of curvature $\ge \kappa$.
\end{theorem}
This is a generalization of Perelman's doubling theorem, which states that for an Alexandrov space $M$ with boundary, its double $D(M)$ is also an Alexandrov space having the same lower curvature bound as $M$, where $D(M)$ is the metric space $M \cup_{\mathrm{id} : \partial M \to \partial M} M$ in the notation used in Theorem \ref{thm:Pet gluing}. 
Petrunin's gluing theorem is also true if the isometry $f$ is defined on components of the boundaries. 

In the present paper, we prove that Alexandrov's lower curvature bound is stable under gluing along {\it certain parts} of boundaries, possibly proper subsets of components of boundaries. 
It gives a generalization of Perelman's and Petrunin's theorems. 

\subsection{Easy examples}
Before stating results, let us observe that various nonnegatively curved surfaces are constructed from one rectangle paper. 
For instance, by gluing two opposite sides of a rectangle, we obtain an annulus or a M\"obius band.
The union $E$ of adjacency sides with induced length metric has a unique non-trivial isometric involution, and the metric space by closing $E$ via the isometry is a nonnegatively curved surface, like a wonky paper cup. 

\[
\begin{matrix}
\begin{matrix}
\includegraphics[bb=0 0 360 270,width=5em, height=4em]{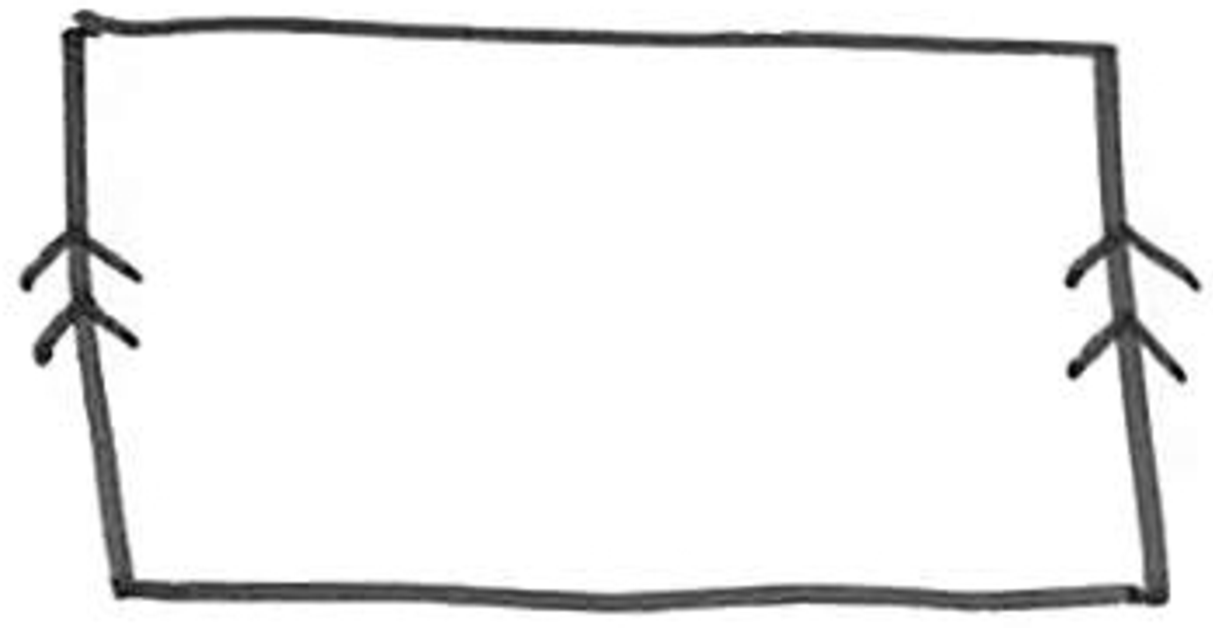} 
\end{matrix} &
\begin{matrix}
\includegraphics[bb=0 0 360 270,width=5em, height=4em]{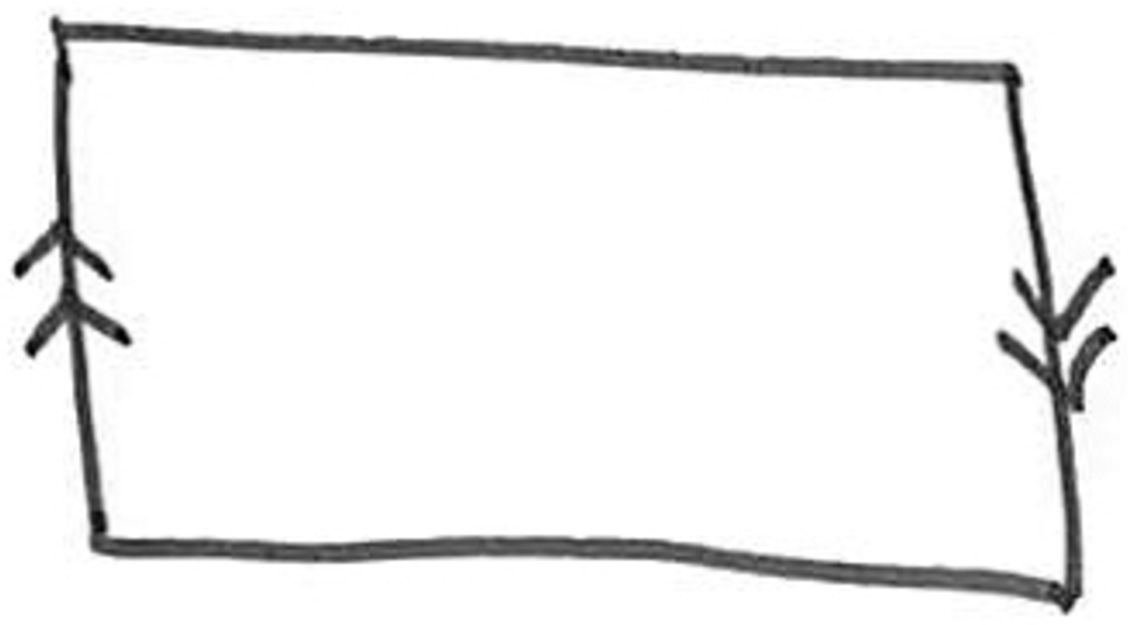}
\end{matrix} &
\begin{matrix}
\includegraphics[bb=0 0 360 270,width=5.6em, height=4em]{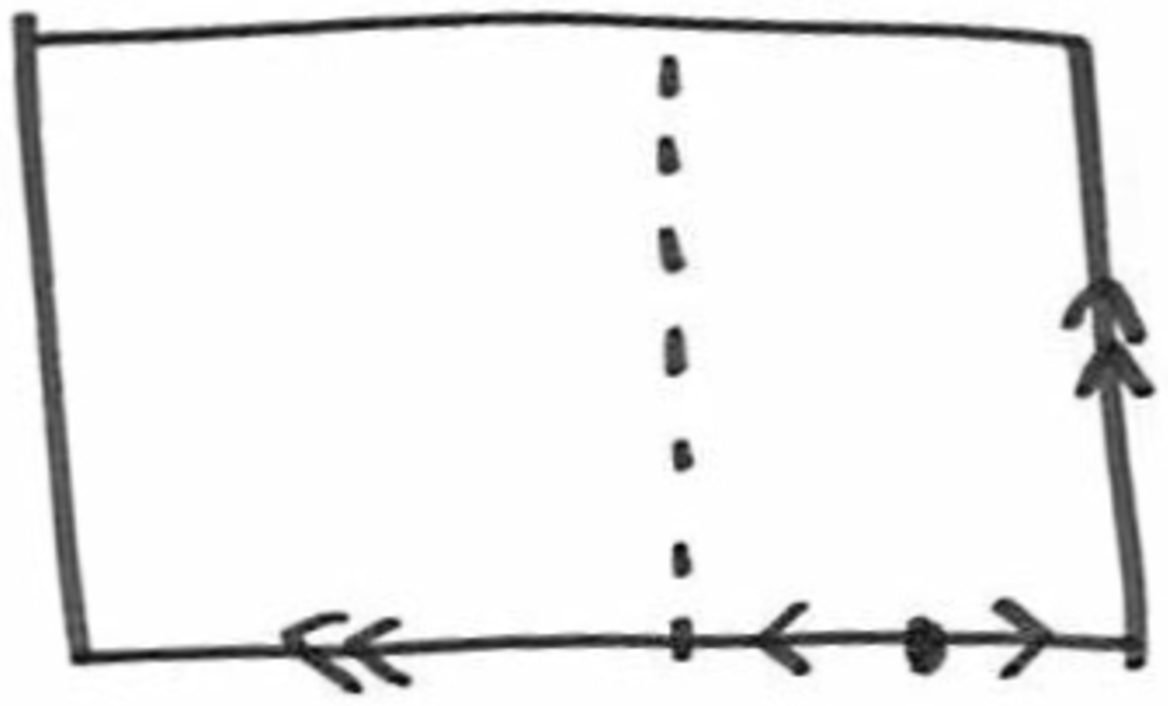}
\end{matrix}
= 
\begin{matrix}
\includegraphics[bb=0 0 360 270,width=5em, height=4em]{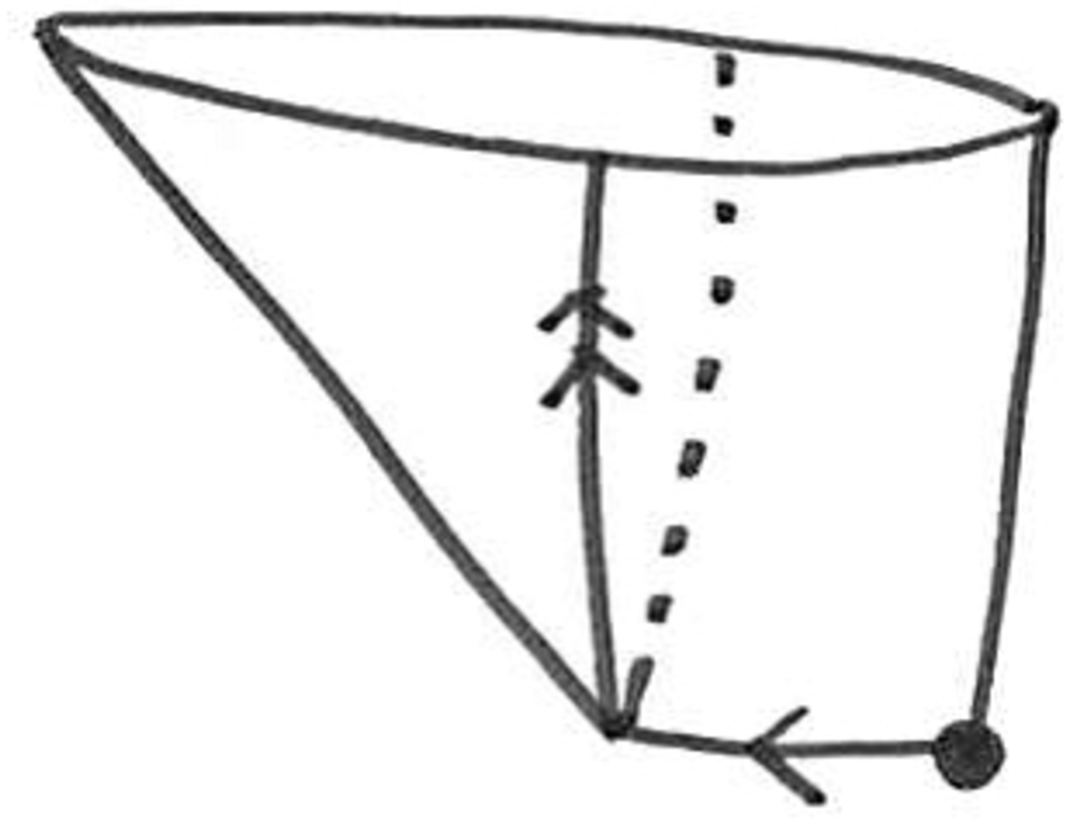} 
\end{matrix} \\
\footnotesize{\text{an annulus}} & \footnotesize{\text{a M\"obius band}} & \footnotesize{\text{a paper cup}}
\end{matrix}
\]

The boundary of a rectangle with the induced length metric is regarded as a circle, which have two kinds of non-trivial isometric involutions $\sigma$ and $\tau$. 
Here, $\sigma$ is a reflection with two fixed points and $\tau$ is a half-rotation of the circle. 
Then, the metric space obtained by closing the boundary of the rectangle via $\sigma$ (resp. via $\tau$) is a sphere (resp. a projective plane) of nonnegative curvature with metric singular points. 
Further, they are flat at almost all points.
For instance, we can obtain the following surfaces: 
\[
\begin{matrix}
&\begin{matrix}
\includegraphics[bb=0 0 360 270,width=5em, height=4em]{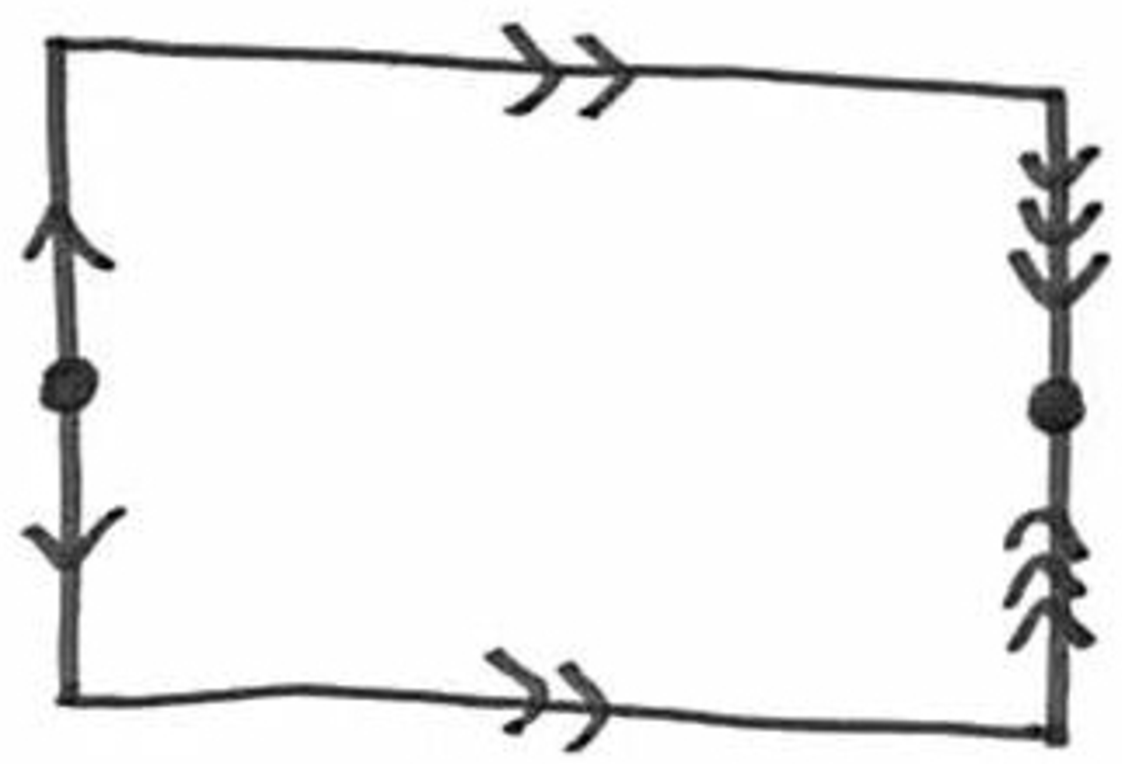}
\end{matrix}
=
\begin{matrix}
\includegraphics[bb=0 0 360 270,width=5em, height=2.3em]{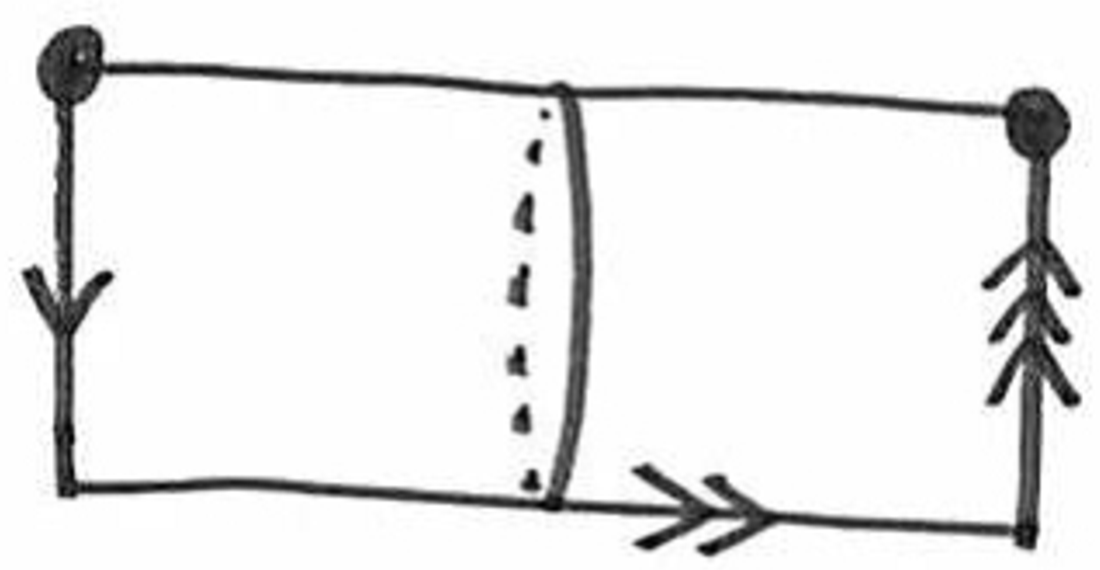}
\end{matrix} 
&\begin{matrix}
\includegraphics[bb=0 0 360 270,width=5em, height=4em]{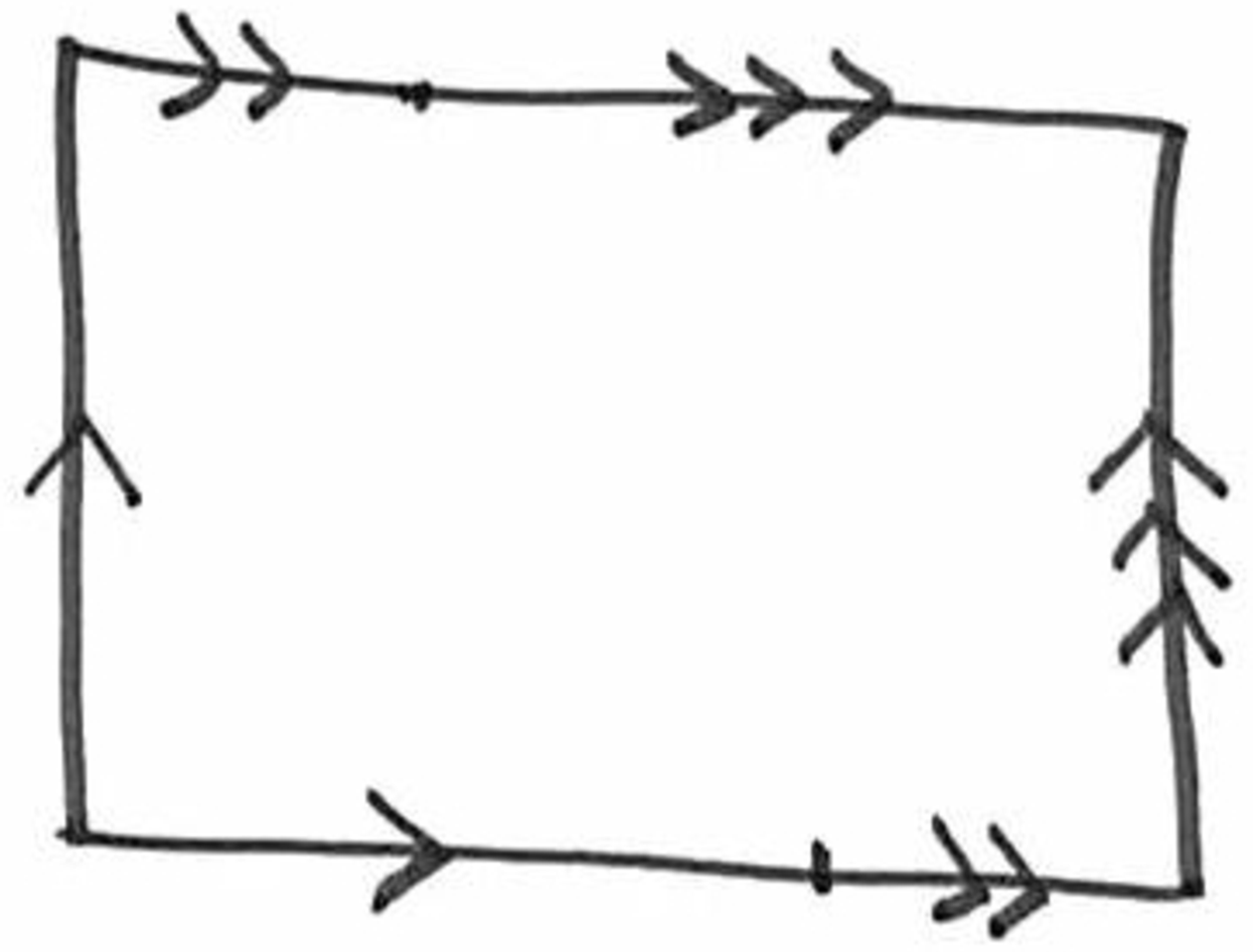}
\end{matrix}
&\begin{matrix}
\includegraphics[bb=0 0 360 270,width=5em, height=4em]{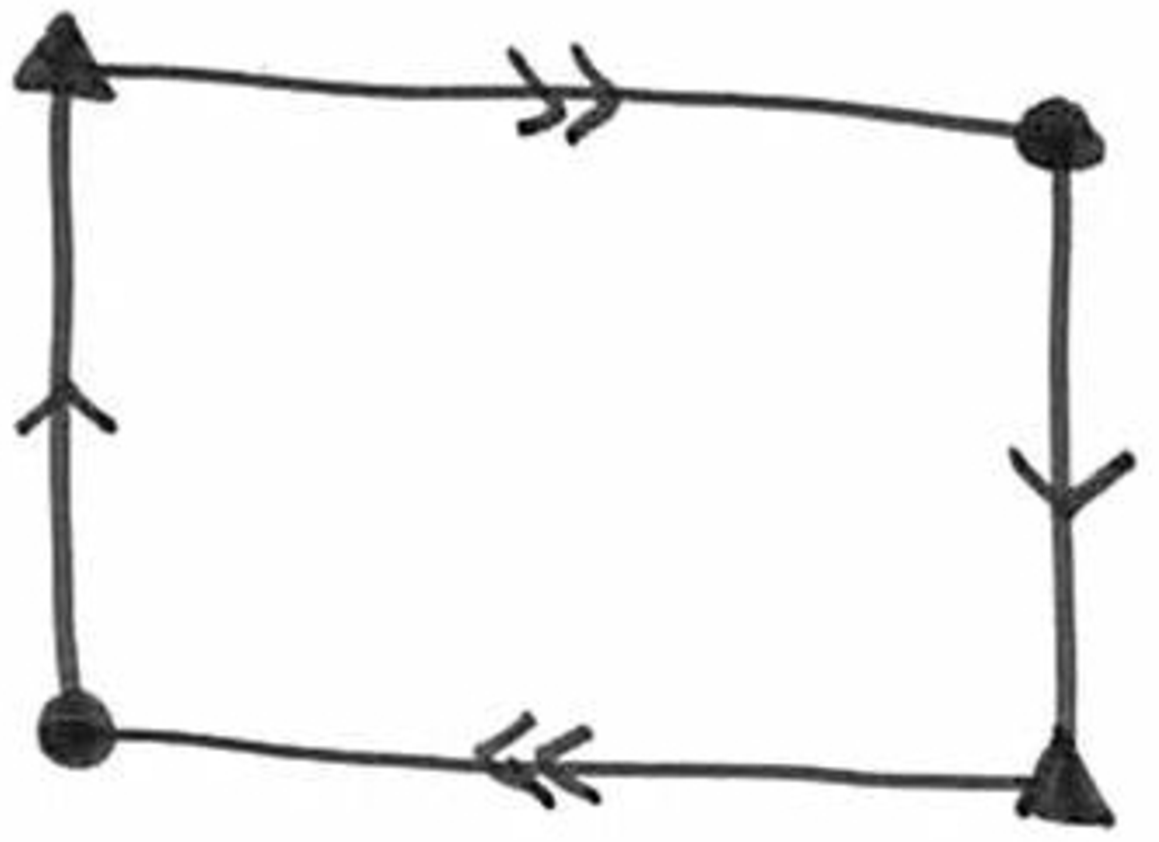}
\end{matrix}
\\
& \footnotesize{\text{a pillow case}} & \footnotesize{\text{a sphere}} \vspace{1em} &\footnotesize{\text{a real projective plane}}
\end{matrix}
\]

Further, by using two rectangles and by gluing sides (or the union of sides) of the same length, we also obtain nonnegatively curved surfaces. 

On the other hands, if we close a segment by an involution on it, where the endpoints of the segment are not corners, then we have a torn envelope: 
\[
\begin{matrix}
\includegraphics[bb=0 0 360 270,width=5em, height=4em]{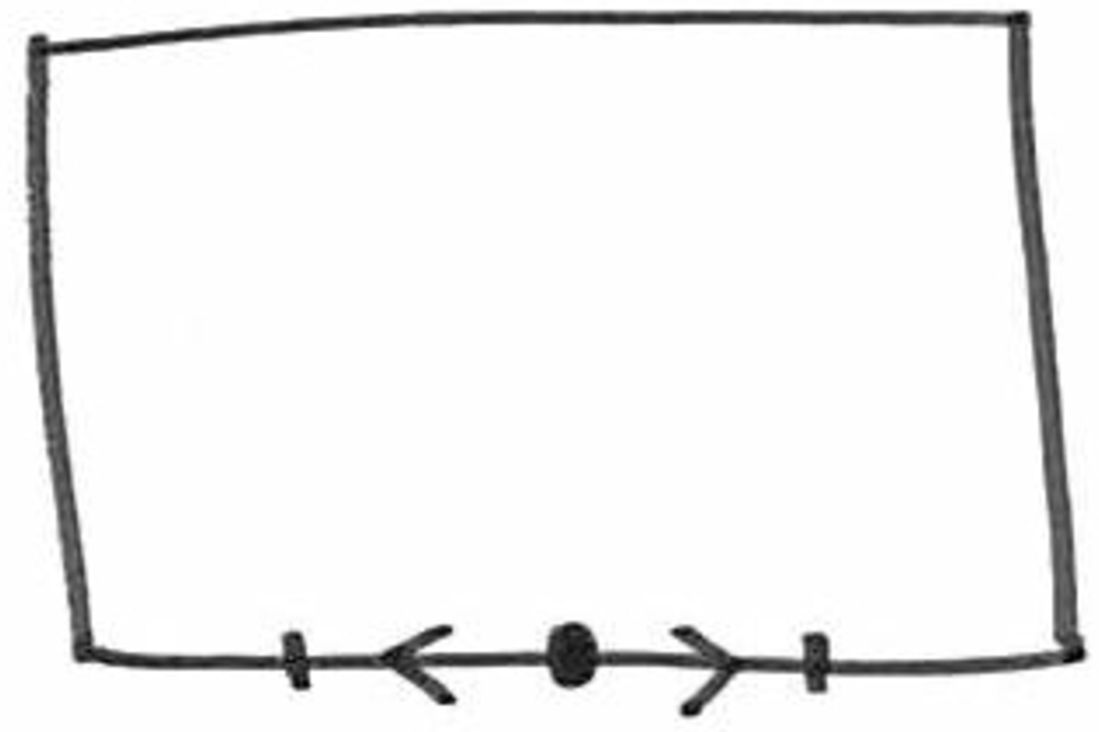}
\end{matrix}
=
\begin{matrix}
\includegraphics[bb=0 0 360 270,width=5em, height=4em]{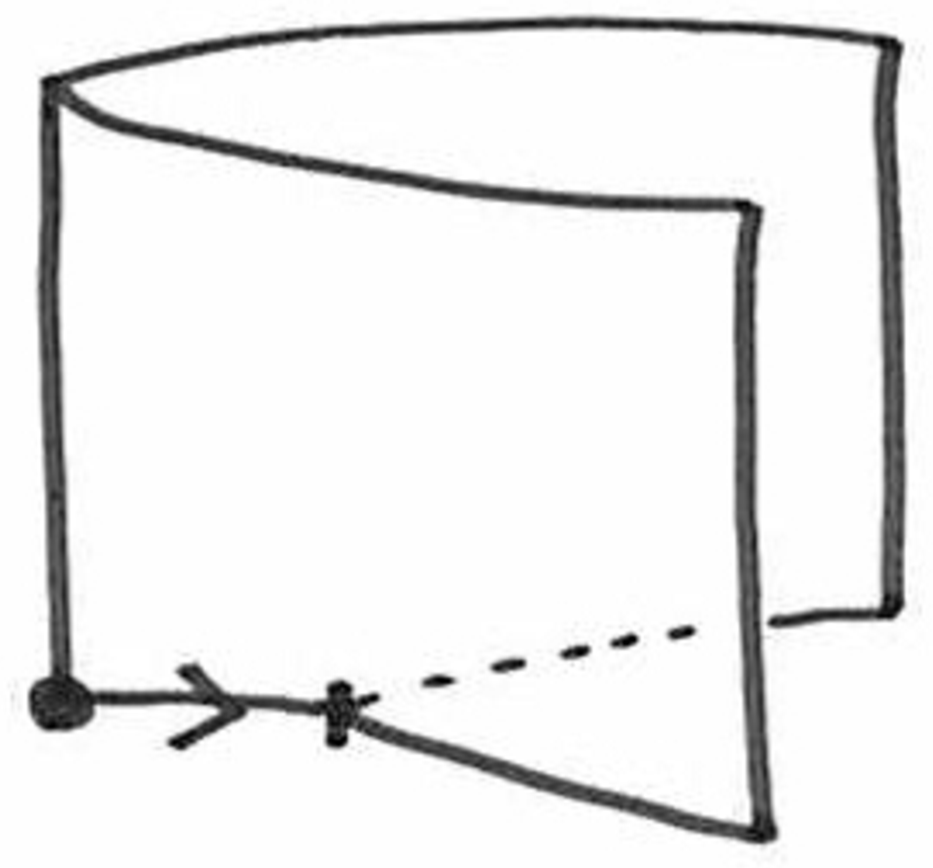}
\end{matrix}
\]
A torn envelope as above does not have a lower curvature bound at the torn point. 

In this paper, we rigorously formulate such phenomena for general Alexandrov spaces, and prove it.

\subsection{Partial gluing}
To state our results, the notion of extremal subsets is needed. 
For example, the union of any set of sides of a rectangle is extremal in the rectangle. 
The precise definition of extremal subsets in general Alexandrov spaces will be reviewed in Section \ref{sec:preliminary}.

Let us state our main results. 

\begin{theorem} \label{thm:self gluing}
Let $M$ be an Alexandrov space which is possibly disconnected and $E$ a codimension one extremal subset of $M$. 
Let $f$ be an isometric involution on $E$ with respect to the induced length metric. 
Then, the metric space $M_f$ obtained by closing $E$ in $M$ via $f$ is an Alexandrov space.  

Further, if $M$ has curvature $\ge \kappa$ for some $\kappa \in \mathbb R$ and $E$ is $\kappa$-extremal in $M$, 
and $M$ has at most two components when $M$ is of dimension one and $\kappa > 0$, 
then $M_f$ is an Alexandrov space of curvature $\ge \kappa$. 
\end{theorem}

Here, the restriction of the number of components is needed when the dimension of $M$ is one and is positively curved (see Remark \ref{rem:dim 1 case}). 
The canonical length metric on $M_f$ in the result of Theorem \ref{thm:self gluing} will be defined in Section \ref{sec:preliminary}.
A direct corollary of Theorem \ref{thm:self gluing} is the following. 

\begin{corollary} \label{cor:partial gluing}
Let $M_1$ and $M_2$ be the same dimensional connected 
Alexandrov spaces of curvature $\ge \kappa$ and $E_1 \subset M_1$ and $E_2 \subset M_2$ their codimension one $\kappa$-extremal subsets. 
Assume that there is an isometry $f : E_1 \to E_2$ in the length metric. 
Then, the metric space $M_1 \cup_f M_2$ obtained by gluing $M_1$ and $M_2$ along $E_1$ and $E_2$ via $f$ is an Alexandrov space of curvature $\ge \kappa$. 
\end{corollary}
This is a generalization of Theorem \ref{thm:Pet gluing}, because Corollary \ref{cor:partial gluing} allows that $E_i$ can be taken to be proper subsests of (a component of) the boundary $\partial M_i$.
We will note that Alexandrov spaces obtained by gluing as in Corollary \ref{cor:partial gluing} naturally appear in the collapsing theory of Riemannin manifolds and Alexandrov spaces with a uniform lower curvature bound (See Example \ref{ex:collapse} and \cite{MY:bdry}). 

\vspace{1em}
\noindent
{\bf Organization.}
The organization of this paper is as follows. 
In \S \ref{sec:preliminary}, we recall basics of length metric spaces, Alexandrov spaces, and their extremal subsets. 
In \S \ref{sec:proof}, we give proofs of Theorem \ref{thm:self gluing} and Corollary \ref{cor:partial gluing}. 

\section{Preliminaries} \label{sec:preliminary}
Let $X = (X, |\cdot, \cdot|)$ denote an abstract metric space possibly having infinite distance. 
Soon, $X$ is restricted to be a length space. 

\subsection{Basics of length spaces}
For a continuous curve $\gamma : [a,b] \to X$ in a metric space $X$, its length is defined as 
\[
L(\gamma) = \sup_{a = t_0 < \dots < t_m= b} \sum_{i=1}^m |\gamma(t_i) \gamma(t_{i-1})|
\]
which is determined independently on the choice of parametrizations of $\gamma$. 
If $\gamma$ is Lipschitz, then its length $L(\gamma)$ is finite. 
For a subset $F \subset X$, the {\it length metric} on $F$ associated with the metric on $X$ is defined as 
\[
|xy|_F := \inf L(\gamma)
\]
for all $x, y \in F$, 
where $\gamma$ runs over continuous curves $\gamma : [0,1] \to F$ with $\gamma(0) = x$ and $\gamma(1) = y$.
If $x$ and $y$ can not be connected by a continuous curve in $F$ of finite length, then $|xy|_F := \infty$.
When $X$ satisfies $|xy| = |xy|_X$ for all $x, y \in X$, it is called a {\it length space}.
Note that for points $x, y$ in a length space, 
they satsify $|xy| < \infty$ if and only if they are contained in the same connected component. 

From now on, $X$ denotes a proper length space possibly having infinite distances. 
Here, a metric space is said to be proper if any closed metric ball is compact.
Then, $X$ becomes automatically a geodesic space. 
Namely, for any $x,y \in X$ wtih $|xy| < \infty$, there is a curve $\gamma : [0,|xy|] \to X$ such that $\gamma(0) = x$, $\gamma(|xy|) = y$ and $L(\gamma) = |xy|$.
Such a $\gamma$ is called a minimal geodesic between $x$ and $y$ and is denoted by $xy$.
In this paper, we always assume that every minimal geodesic is parametrized by the arclength.
That is, any minimal geodesic is just an isometric embedding from an interval to a metric space. 

Let $\Sigma$ be a metric space. 
For $\kappa \in \mathbb R$, the {\it $\kappa$-cone} $C^\kappa \Sigma$ over $\Sigma$ is defined as follows.
Let $D_\kappa := \pi / \sqrt \kappa$ if $\kappa > 0$ and  $D_\kappa := + \infty$ if $\kappa \le 0$. 
Let a function $\mathrm{sn}_\kappa : [0, D_\kappa) \to \mathbb R$ be the unique solution to an ODE
\[
\mathrm{sn}_\kappa'' + \kappa\, \mathrm{sn}_\kappa = 0, \mathrm{sn}_\kappa(0) = 0, \mathrm{sn}_\kappa'(0) = 1.
\]
We set $\mathrm{cs}_\kappa(t) = \mathrm{sn}_\kappa'(t)$.
Consider the product $[0, D_\kappa / 2) \times \Sigma$. 
For two points $(a,\xi)$, $(b,\eta) \in [0,D_\kappa/2) \times \Sigma$, we define the $\kappa$-cone distance $|(a,\xi),(b,\eta)|$ by
\[
\csk |(a,\xi),(b,\eta)| = \csk a \, \csk b + \kappa\, \snk a \, \snk b \cos (\min \{|\xi \eta|, \pi\})
\]
if $\kappa \neq 0$, and by 
\[
|(a,\xi),(b, \eta)|^2 = a^2 + b^2 - 2 a b \cos (\min \{|\xi \eta|, \pi\})
\]
if $\kappa = 0$.
The $0$-cone distance is equal to the limit of $\kappa$-cone distance as $\kappa \to 0$.
The metric completion of $[0,D_\kappa /2) \times \Sigma$ with respect to the $\kappa$-cone distance is denoted by $C^\kappa \Sigma$, which is the $\kappa$-cone over $\Sigma$.
In particular, all points in $\{0\} \times \Sigma$ are identified as a one point in $C^\kappa \Sigma$, which is called the vertex of the $\kappa$-cone and is denoted by $o$.
A point in $C^\kappa \Sigma$ is called a vector, and if it is represented by $(a,\xi)$, then it is written as $a \xi$.
The norm $|v|$ of a vector $v = a \xi$ is defined as $|v| := |v o| = a$.
The $0$-cone is usually called the Euclidean cone, and the $1$-cone is the spherical join to a one-point space.
If $\Sigma$ is a connected length metric space of diameter at most $\pi$ or a space consisting of two points of distance $\pi$, then its $\kappa$-cone $C^\kappa \Sigma$ becomes a length metric space. 

\subsection{Canonical length metrics on glued spaces} \label{subsec:length}
Let $X = (X, |\cdot,\cdot|_X)$ be a proper geodesic space and $E$ its closed subset. 
We denote by $|\cdot,\cdot|_E$ the length metric on $E$ induced from the original metric $|\cdot,\cdot|_X$. 
Let us impose the following condition ($\spadesuit$) on $(X,E)$. 
\begin{itemize}
\item[($\spadesuit$)] The original metric and the length metric on $E$ are locally bi-Lipschitz, that is, for each $x \in E$, there exist $r> 0$ and $C \ge 1$ such that $|yz|_X \le |yz|_E \le C |yz|_X$ for all $y,z \in B(x,r) \cap E$, where $B(x,r)$ is the closed ball in the original metric. 
\end{itemize}
Further, let us fix an isometric involution $f$ on $(E, |\cdot,\cdot|_E)$, where $f$ may be trivial. 

Let us consider an equivalent relation on $X$ generated by $E \ni x \sim f(x) \in E$, and the quotient $X_f := X /\!\! \sim$ as a set. 
Let $\pi : X \to X_f$ be the projection. 
In \cite[\S 3]{BBI}, it was introduced that an equivalent relation on a general metric space induces a psuedo-distance on the space as follows. 
We use such a procedure in our case. 
For $x, y \in X$, we set 
\begin{equation} \label{eq:length}
|xy| = \inf \sum_{i=0}^k |x_i y_i|_X
\end{equation}
where the infimum runs over all choices of $\{x_i\}$ and $\{y_i\}$ such that $x_0 = x$, $y_k =y$, $k \in \mathbb Z_{\ge 0}$ with $y_{i-1} \sim x_i$ for all $i = 1, \dots, k$.

\begin{lemma} \upshape \label{lem:length}
Let $X$ be a proper geodesic space and $E$ a closed subset satisfying ($\spadesuit$). 
Let $f$ be an isometric involution on $E$ with respect to the length metric. 
Then, we have 
\begin{enumerate}
\item The equivalent relation on $X$ defined as $|xy|=0$ is equal to the equivalent relation $\sim$ induced from $(E,f)$, that is, $|xy|=0$ if and only if $x \sim y$ for all $x, y \in X$.
\item The topology on $X_f$ induced from the length metric \eqref{eq:length} coincides with the quotient topology on $X_f$.
\item The metric space $(X_f, |\cdot,\cdot|)$ is a proper geodesic space.
\end{enumerate}
\end{lemma}
\begin{proof}
From the definition, $|\cdot,\cdot| \le |\cdot,\cdot|_X$ holds. 
Further, for every $x \in X \setminus E$, letting $r = |xE|$, the restriction of two distances $|\cdot,\cdot|_X$ and $|\cdot,\cdot|$ to $U(x,r)$ coincide, where $U(x,r)$ denotes the open ball centered $x$ of radius $r$ in the original metric. 
From the assumption ($\spadesuit$), for any $x \in E$, there exist $r > 0$ and $C \ge 1$ such that 
\begin{equation} \label{eq:open map}
U^f(\{x, f(x)\},r) \subset U(\{x,f(x)\}, Cr)
\end{equation}
holds, where $U^f(y,s)$ denotes the open ball centered at $y$ of radius $s$ with respect to the induced distance $|\cdot,\cdot|$. 
Indeed, by the condition ($\spadesuit$), there exist $r > 0$ and $C \ge 1$ such that for all $y,z \in U(\{x,f(x)\},r) \cap E$, we have $|yz|_E \le C|yz|_X$. 
Let $y \in U^f(x, r)$. 
Then, there exist sequences of points $\{x_i\}, \{y_i\}$ such that $x_i, y_i \in E$ for all $1 \le i \le k-1$ and $y_{i-1} \sim x_i$ for all $1 \le i \le k$, and that $x=x_0$, $y=y_k$ and 
\[
\sum_{i=0}^k |x_iy_i|_X < r
\]
hold. 
If $y_{i-1} = x_i$ for every $1 \le i \le k$, then by the triangle inequality, we have $|xy|_X < r$. 
So, we may assume that there is $i$ with $1 \le i \le k$ such that $y_{i-1} \neq x_i$. 
Then, $k \ge 1$. 
We only consider the case $y_0 \neq x_1$ and $y_1 = y$.  
Hence, we have 
\[
|x_0 y_0|_X + |x_1y_1|_X < r. 
\]
We note that 
\[
|f(x_0)f(y_0)|_X \le C |f(x_0)f(y_0)|_E = C |x_0y_0|_E \le C|x_0y_0|_X
\]
Therefore, we have 
\[
|f(x)y|_X \le |f(x_0)f(y_0)|_X + |x_1 y_1|_X < C r.
\]
Thus, we obtain \eqref{eq:open map}.
These imply (1), (2) and (3). 
\end{proof}

Let $\pi : X \to X_f$ be the quotient map. 
From now on, each point $x \in X_f$ is considered as the subset $x = \pi^{-1}(x) \subset X$ 
and the image of $E$ under the projection $\pi$ is denoted by $E_f \subset X_f$. 
Further, for a point $x \in X$, we use $x$ the same symbol to indicate its equivalent class $x = \pi(x) \in X_f$. 
Under this convention, as in \cite{Pet:appl}, we consider a concept to approximate the distance defined in \eqref{eq:length}. 
For $x,y \in X_f$ and $m \in \mathbb Z_{\ge 0}$, the $m$-{\it predistance} between $x$ and $y$ is defined by 
\[
|xy|_m = \inf \sum_{i=0}^k |p_i p_{i+1}|_X
\]
where the infimum runs over points $\{p_i\}_{i=0}^{k+1}$ satisfying $p_0=x$, $p_{k+1}=y$, $p_i \in E_f$ for $i \in \{1, \dots, k\}$ with $k \le m$.
Here, the right-hand side is the sum of distances between sets in $X$. 

\begin{lemma} \label{lem:predist}
Let $x,y \in X_f$. Then, we have the following. 
\begin{enumerate}
\item $|xy|_m \ge |xy|_{m+1}$. 
\item $|xz|_m + |zy|_\ell \ge |xy|_{m+\ell}$ for $z \in X_f$. 
Further, if $z \in E_f$, then $|xz|_m + |zy|_\ell \ge |xz|_{m+\ell+1}$. 
\item $|xy|_m$ converges to $|xy|$ as $m \to \infty$. 
\item For $m \in \mathbb Z_{\ge 0}$, if $|xy|_m$ is finite, then there is a sequence of points $\{p_i\}_{i=0}^{k+1}$ with $k \le m$ which attains the infimum in the definition of $m$-distance $|xy|_m$. 
\end{enumerate}
\end{lemma}
\begin{proof}
The properties (1), (2) and (3) trivially follow from the definition. 
The property (4) follows from the properness of $X$.
\end{proof}
A sequence $\{p_i\}_{i=0}^{k+1}$ obtained as Lemma \ref{lem:predist} (4) is called an $m$-{\it shortest path} between $x$ and $y$. 

We use the following convention: for two proper geodesic spaces $X_1$ and $X_2$ and their closed subsets $E_1 \subset X_1$ and $E_2 \subset X_2$ satisfying ($\spadesuit$), if there is an isometry $f : E_1 \to E_2$ with respect to the induced length metrics, then a space $X_1 \cup_f X_2$ denotes the metric space $X_{\bar f}$, where $X$ is the disjoint union of $X_1$ and $X_2$ and $\bar f$ is the canonical extension of $f$ to an isometric involution on $E_1 \sqcup E_2$ defined as $\bar f(x) = f^{-1}(x)$ for $x \in E_2$.
This is the definition of the metric on the space $M_1 \cup_f M_2$ in Corollary \ref{cor:partial gluing}.

\subsection{Basics of Alexandrov spaces}
Let $A,B,C$ be nonnegative numbers satisfying triangle inequality $A + B \ge C \ge |A-B|$.
When $\kappa \neq 0$, the $\kappa$-comparison angle of $(A;B,C)$ is defined as follows. 
\[
\tilde \angle_\kappa (A;B,C) := 
\arccos \frac{\csk A - \csk B\, \csk C}{\kappa\, \snk B\, \snk C} \in [0,\pi]
\]
if $B \cdot C > 0$ and $A+B+C < 2 D_\kappa$. 
Otherwise, $\tilde \angle_\kappa (A;B,C) = 0$. 
Further, we set
\[
\tilde \angle_0 (A;B,C) := \lim_{\kappa \to 0} \tilde \angle_\kappa (A;B,C). 
\]
For points $a,b,c$ in a metric space, we set 
\[
\tilde \angle_\kappa b a c := \tilde \angle_\kappa (a;b,c) := \tilde \angle_\kappa (|bc|; |ab|, |ac|)
\]
which is called the $\kappa$-comparison angle of $\{a,b,c\}$ at $a$.

\begin{definition} \upshape \label{def:Alex}
A proper geodesic space $M$ which is possibly disconnected, is called an {\it Alexandrov space} if for any $x \in M$, there exist a number $\kappa \in \mathbb R$ and a neighborhood $U$ of $x$ in $M$ such that for any distinct four points $a,b,c,d$ in $U$, we have the comparison inequality 
\[
\tilde \angle_\kappa b a c + \tilde \angle_\kappa c a d + \tilde \angle_\kappa d a b \le 2 \pi.
\]
A number $\kappa$ as above is called a lower curvature bound at $x$ in $X$. 
When $\kappa$ is taken independently on the choice of points $x$, we say that $M$ is an Alexandrov space of curvature $\ge \kappa$. 
\end{definition}
Due to the globalization theorem (\cite{BGP}), if $M$ is a connected Alexandrov space of curvature $\ge \kappa$ with $\kappa > 0$, then any geodesic triangle has perimeter not greater than $2 D_\kappa$. 
In particular, each component of $M$ has diameter at most $D_\kappa$. 

For each component $M$ of an Alexandrov space, its Hausdorff dimension is known to be same as the topological dimension. 
The same value is called the dimension of $M$.
We say that an Alexandrov space is of dimension $n$, if every component is of dimension $n$.
In the present paper, we deal with only finite dimensional Alexandrov spaces. 

There are other characterizations of Alexandrov spaces based on several conditions, called, triangle comparison, hinge comparison and comparison-angle monotonicity (see \cite{BGP}, \cite{BBI}). 
Such conditions are generalized for $1$-Lipschitz curves instead of geodesics (see Proposition \ref{prop:convex curve}) in a suitable way. 
So, the definitions of such conditions are omitted, here.  

From now on, we denote by $M$ an Alexandrov space of finite dimension. 
For $x \in M$ and two geodesics $\alpha, \beta : [0,a] \to M$ starting at $x = \alpha(0) = \beta(0)$, 
the comparison-angle monotonicity enable to us define the {\it angle} between $\alpha$ and $\beta$ at $x$ as 
$\angle (\alpha, \beta) := \lim_{s,t \to 0} \tilde \angle_0 \alpha(s) x \beta(t)$.
On the set of all non-constant geodesics emanating from $x$ in $M$, the angle becomes a pseudo-distance. 
The metric space of all equivalent classes of geodesics obtained in a usual way is denoted by $\Sigma_x' = \Sigma_x' M$. 
Its completion $\Sigma_x = \Sigma_x M$ is called the space of directions at $x$. 
Further, $\Sigma_x M$ becomes a connected Alexandrov space of curvature $\ge 1$ and of dimension $(\dim M -1)$ or the metric space consisting of two points with the distance $\pi$.
In particular, the diameter of $\Sigma_x$ is not greater than $\pi =D_1$. 
For $y \neq x$, we denote by $y_x' \subset \Sigma_x' M$ the set of all directions of geodesics from $x$ to $y$. 
An element of $y_x'$ is denoted by $\uparrow_x^y$.

For the definition of Gromov-Hausdorff convergence, we refer \cite{BBI}. 
The Gromov-Hausdorff limit of the family of scaling pointed spaces $(r M, x)$ as $r \to \infty$, where $r M$ denotes the set $M$ equipped with the distance function of $M$ multiplied with $r$, always exists and is denoted by $(T_x M, o_x)$, which is called the tangent cone of $M$ at $x$. 
It is known that $(T_x M, o_x)$ is isometric to the Euclidean cone over $\Sigma_x$ with the vertex. 
By the definition, $T_x M$ has nonnegative curvature.

For each $x \in M$, we define the (multi-valued) logarithm map at $x$ as 
\[
\log_x : M \ni y \mapsto |x y| y_x' \subset T_x M,
\]
where, we set $\log_x(x) := o_x$.
The {\it exponential map} at $x$ is defined as the left-inverse of $\log_x$ which is single-valued, that is, $\exp_x \circ \log_x = \mathrm{id}_M$.

The $\kappa$-cone over $\Sigma_x M$ is denoted by $T_x^\kappa M$. 
Since a map $\log_x^\kappa : B(x,D_\kappa/2) \to T_x^\kappa M$ can be defined in the same manner to define the usual logarithm as above, 
the corresponding exponential map $\exp_x^\kappa$ can be defined as the map from some subset of $T_x^\kappa M$ such that $\exp_x^\kappa \circ \log_x^\kappa = \mathrm{id}_{B(x,D_\kappa/2)}$.
If $\kappa$ is a lower curvature bound on $M$, then $\exp_x^\kappa$ is $1$-Lipschitz.
Note that the image of $\log_x^\kappa$ is not dense in $T_x^\kappa M$. 
For instance, if $M$ is a closed convex region with boundary in the plane, 
then a vector in the tangent cone $T_x^0 M$ at $x \in M$ can not be approximated by the image of $\log_x^0$, 
indeed, the image of $\log_x^0$ is an isometric copy of $M$ in $T_x^0 M$. 
However, $\exp_x^\kappa$ is extended to a $1$-Lipschitz map defined on the whole space $T_x^\kappa M$ (see Subsection \ref{subsec:grad}).

\subsection{The boundary}
The {\it boundary} $\partial M$ of an Alexandrov space $M$ is defined as follows. 
If $\dim M = 1$, then $M$ is a complete one-manifold possibly with boundary and $\partial M$ is the usual boundary. 
When $\dim M \ge 2$, for $p \in M$, we say $p \in \partial M$ if $\partial \Sigma_p \neq \emptyset$.

\subsection{Semiconcave functions on Alexandrov spaces}

Let $M$ denote an Alexandrov space.

\begin{definition} \upshape \label{def:semiconcave}
Let $U$ be an open subset of $M$ and $g : U \to \mathbb R$ a function. 
A locally Lipschitz function $f : U \to \mathbb R$ is said to be {\it $g$-concave} (in the barrier sense) 
if for any $x \in U$ and $\epsilon > 0$, there exists a neighborhood $V$ of $x$ in $U$ such that for any minimal geodesic $\gamma$ contained in $V$, the function $f$ along $\gamma$ is $(g(x)+\epsilon)$-concave in the sense that 
\[
f \circ \gamma(t) - \frac{g(x)+\epsilon}{2} t^2
\]
is concave in $t$.
In this case, we use the notation 
\[
f'' \le g \text{ on } U.
\]
A locally Lipschitz function which is $g$-concave for some $g$ is called a {\it semiconcave function}.
\end{definition}

Recall that any distance function $d_A$ from a compact subset $A$ of an Alexandrov space $M$ is semiconcave. 
Indeed, if $U$ is an open subset containing $A$ and has a lower curvature bound $\kappa$, 
then we have 
\begin{equation} \label{eq:rho dist}
(\rho_\kappa \circ d_A \circ \gamma(t))'' \le 1 - \kappa (\rho_\kappa \circ d_A \circ \gamma(t)) 
\end{equation}
for each geodesic $\gamma$ in $U$.
Here, $\rho_\kappa$ is defined as 
\[
\rho_\kappa(u) = \int_0^u \snk\, v\, d v.
\]
When $\kappa \neq 0$, this is represented as $\rho_\kappa = (- \kappa)^{-1} (\csk - 1)$.
Hence, setting $h := - \kappa^{-1} \csk (d_A \circ \gamma(t))$, the condition \eqref{eq:rho dist} is equivalent to 
\begin{equation} \label{eq:rho dist equiv}
h'' + \kappa h \le 0
\end{equation}
in the barrier sense, if $\kappa \neq 0$.

From a smoothing argument and a direct calculation, we have 
\begin{lemma} \label{lem:sine comp}
Let $I$ be an interval and $f : I \to [0, D_\kappa]$ a Lipschitz function satisfying 
\[
(\rho_\kappa \circ f)'' + \kappa \rho_\kappa \circ f \le 1
\]
in the barrier sense. 
Then, in the domain $\{t \mid 0 < f(t) < D_\kappa / 2 \}$, the function $k := \snk f$ satisfies 
\[
k''(t) \le \frac{\csk^2 f(t) - A_t^2}{\snk f(t)}
\]
in the barrier sense, where 
\[
A_t = \limsup_{t' \to t} \left|\frac{f(t')-f(t)}{t'-t}\right|.
\]
\end{lemma}

From a basic calculus of one variable, we obtain the following Lemma \ref{lem:contradiction1}. 
This is needed in the proof of Theorem \ref{thm:self gluing}. 
For the completeness, we prove Lemma \ref{lem:contradiction1}. 

\begin{lemma} \label{lem:contradiction1}
Let $f : [a,b] \to \mathbb R$ be a continuous function which is concave in the first order, that is, it satifies 
\begin{equation} \label{eq:a1}
\overline f^+(t) \le \underline f^-(t)
\end{equation} 
for all $t \in (a,b)$, where $\overline f^+(t) = \limsup_{\epsilon \to 0+} (f(t+\epsilon)-f(t))/\epsilon$ and $\underline f^-(t) = \liminf_{\epsilon \to 0+} (f(t)-f(t-\epsilon))/\epsilon$. 
Let $g : [a,b] \to \mathbb R$ be a continuous function. 
Suppose that $f$ does not satisfy $f'' \le g$ on $(a,b)$ in the barrier sense. 
Then, there exist $t_0 \in (a,b)$, $A \in \mathbb R$ and $\epsilon > 0$ such that 
\[
f(t) \ge f(t_0) + \frac{g(t_0)}{2} (t-t_0)^2 + \epsilon (t-t_0)^2 + A (t-t_0)
\]
for all $t$ with $|t-t_0|<\epsilon$. 
\end{lemma}
\begin{proof}
From the assumption, there exist $\epsilon >0$ and $t_1 \in (a,b)$ such that a function
\[
\varphi(t; t_1) := f(t) - \frac{g(t_1) + \epsilon}{2} t^2
\]
is not concave in $t$ on any neighborhood $I$ of $t_1$. 
Hence, there exists an interval $(c,d) \subset I$ such that 
\[
\varphi(t;t_1) < \ell(t)
\]
holds on $(c,d)$, where $\ell$ is a linear function satisfying $\ell(c)=\varphi(c;t_1)$ and $\ell(d)=\varphi(d;t_1)$. 
Let $t_0 \in (c,d)$ be a minimizer of $\varphi(t;t_1) -\ell(t)$. 
Then, we have 
\begin{equation} \label{eq:a2}
\varphi(t;t_1) \ge \varphi(t_0;t_1) + \ell(t) - \ell(t_0)
\end{equation}
on $[c,d]$. 
Note that the assumption \eqref{eq:a1} implies that $\varphi$ and $f$ is differentiable at $t_0$. 
Hence, we have
\[
\ell'(t_0) = \varphi'(t_0;t_1) = f'(t_0) - (g(t_1) + \epsilon)t_0.
\]
Therefore, the equation \eqref{eq:a2} is represented as 
\begin{align*}
f(t) - f(t_0) 
&\ge
\frac{g(t_1)+\epsilon}{2}(t -t_0)^2 + f'(t_0) (t-t_0).
\end{align*}
Since $g$ is uniformly continuous, we may assume that $|g(t_0)-g(t_1)| \le \epsilon/2$ by taking $I$ to have small length. 
Then, we obtain 
\[
f(t)-f(t_0) \ge \frac{g(t_0)+\epsilon/2}{2} (t-t_0)^2 + f'(t_0) (t-t_0)
\]
on $[c,d]$. 
This complete the proof of Lemma \ref{lem:contradiction1}.
\end{proof}

\subsection{Gradient curves for semiconcave functions} \label{subsec:grad}
For a semiconcave function $f : U \to \mathbb R$ and $x \in U$, the {\it derivative} $f' = f_x'$ of $f$ at $x$ is defined a function on $\Sigma_x'$ defined by 
\[
f_x'(\xi) := \lim_{t \to 0+} \frac{f(\exp_x(t \xi))- f(x)}{t}.
\]
Then, $f_x'$ becomes a Lipschitz function on $\Sigma_x'$. 
Hence, there is a unique Lipschitz extension on $\Sigma_x$ of $f_x'$. 
We denote it by the same symbol $f_x' : \Sigma_x \to \mathbb R$.
Then, $f_x'$ is spherically concave, that is, it is $(- f_x')$-concave.
The cone extension of $f_x'$ is also denoted by $f_x' : T_x M \to \mathbb R$ 
defined as $f_x'(a \xi) = a f_x' (\xi)$ for $a \ge 0$ and $\xi \in \Sigma_x$, which is $0$-concave.

The {\it gradient vector} of a semiconcave function $f : U \to \mathbb R$ at $x \in U$ 
is a vector $g \in T_x M$ satisfying $|g|^2 = f_x'(g)$ and $\left< g, v \right> \ge f_x'(v)$ for all $v \in T_x M$.
It always exists and is unique, and is denoted by $g = \nabla f(x)$.
Note that $f_x' \le 0$ everywhere if and only if $\nabla f(x) = o_x$.

For a locally Lipschitz curve $\gamma : [0,a) \to U$, the {\it forward direction} $\gamma^+(t) \in T_{\gamma(t)} M$ at $t \in [0,a)$ 
is defined as follows. 
When the limit $\lim_{\epsilon \to 0+} |\gamma(t) \gamma(t+\epsilon)| / \epsilon$ exists and is zero, we set $\gamma^+(t) = 0$. 
When the limit $\lim_{\epsilon \to 0+} |\gamma(t)\gamma(t+\epsilon)| / \epsilon$ exists and is positive, 
and the limit of a sequence of directions $\uparrow_{\gamma(t)}^{\gamma(t+\epsilon_i)}$ is uniquely determined independently on the choices of $\epsilon_i \to 0+$ and geodesic directions $\uparrow_{\gamma(t)}^{\gamma(t+\epsilon_i)} \in \gamma(t +\epsilon_i)_{\gamma(t)}'$, 
we set 
\[
\gamma^+(t) := \lim_{i \to \infty} \frac{|\gamma(t)\gamma(t+\epsilon_i)|}{\epsilon_i} \uparrow_{\gamma(t)}^{\gamma(t+\epsilon_i)}.
\]
Similarly, the backward direction at $t$ is defined to be $\gamma^-(t) := \gamma^+(-s + t) |_{s = +0}$ if it exists.

\begin{definition}[{\cite{PP:qg}, \cite{Pet:qg}, \cite{Pet:semi}}] \upshape \label{def:gradient curve}
Let $f : U \to \mathbb R$ be a semiconcave function. 
A {\it gradient curve} for $f$ starting at $x \in U$ is a locally Lipschitz curve $\gamma : [0,a) \to U$ with $\gamma(0) = x$ such that the forward direction $\gamma^+(t)$ always exists at any time $t \in [0,a)$ and satisfies 
\[
f_{\gamma(t)}'(\gamma^+(t)) = |\nabla f(\gamma(t))|^2, 
\]
equivalently, 
\[
\gamma^+(t) = \nabla f(\gamma(t)).
\]
\end{definition}

For any semiconcave function, its gradient curve starting at any point is known to be uniquely exists (\cite{PP:qg}, \cite{Pet:qg}, \cite{Pet:semi}). 

Modifying gradient flows of distance functions, the following concept was defined in \cite{PP:qg} and \cite{Pet:semi}. 
For every vector $\xi \in \Sigma_x$, we define a curve $\alpha_\xi : [0,D_\kappa/2] \to M$ such that $\alpha_\xi(0) = x$ and 
\[
\alpha_\xi^+(t) = \frac{\mathrm{tg}_\kappa(|x \alpha_\xi(t)|)}{\mathrm{tg}_\kappa(t)} \nabla d_x (\alpha_\xi(t))
\]
holds for every $t \in (0,D_\kappa/2)$.
Here, $\mathrm{tg}_\kappa (u) = \snk(u) / \csk(u)$.
The $\kappa$-{\it gradient-exponential map} $\mathrm{gexp}_x^\kappa : T_x^\kappa M \to M$ is defined as 
\[
\mathrm{gexp}_x^\kappa(t \xi) := \alpha_{\xi}(t)
\]
for all $t \xi \in T_x^\kappa M$.
Note that if $\exp_x^\kappa(v)$ is defined, then $\mathrm{gexp}_x^\kappa(v) = \exp_x^\kappa(v)$ for $v \in T_x^\kappa M$.

\begin{theorem}[\cite{PP:qg}, \cite{Pet:semi}]
Let $M$ be an Alexandrov space of curvature $\ge \kappa$.
Then, the $\kappa$-gradient-exponential map $\gexp_x^\kappa : T_x^\kappa M \to M$ is $1$-Lipschitz and satisfying that 
$\gexp_x^\kappa \circ \log_x = \mathrm{id}$ on $B(x,D_\kappa/2)$. 
In particular, the image of $\gexp_x^\kappa$ is $B(x,D_\kappa/2)$.
\end{theorem}

\subsection{Quasigeodesics}

\begin{definition}[\cite{PP:qg}, \cite{Pet:qg}] \upshape \label{def:kappa-convex}
A $1$-Lipschitz curve $\gamma$ in a proper geodesic space $X$ is called $\kappa$-{\it convex} in $X$ 
if for any parameter $t_0$ of $\gamma$ and $x \in X$ close to $\gamma(t_0)$, there is a neighborhood $I$ of $t_0$ such that 
\[
(\rho_\kappa (|x \gamma(t)|))'' \le 1 - \kappa \rho_\kappa (|x \gamma(t)|)
\]
holds in the barrier sense on $I$.
\end{definition}

\begin{proposition}[{\cite[\S1.4 and 1.5]{PP:qg}}] \label{prop:convex curve}
For a $1$-Lipschitz curve $\gamma : [a,b] \to X$ in a metric space $X$. 
The following conditions are equivalent. 
\begin{itemize}
\item[(0)] 
$\gamma$ is $\kappa$-convex. 
\item[(1)] 
Let $p \in X$, $q_1 = \gamma(t_1)$, $q_2 = \gamma(t_2)$ and $q_3 = \gamma(t_3)$ 
with $t_1<t_2<t_3$ and $t_3-t_1 \le |pq_1| + |pq_3| < 2 D_\kappa - (t_3-t_1)$. 
Let $\tilde p, \tilde q_1, \tilde q_2, \tilde q_3 \in \mathbb M_\kappa$ with $|\tilde p \tilde q_1| = |p q_1|$, $|\tilde p \tilde q_3| = |pq_2|$, $|\tilde q_1 \tilde q_2| = t_2-t_1$, $|\tilde q_3 \tilde q_1| = t_3-t_1$ and $|\tilde q_3 \tilde q_2| = t_3 -t_2$. 
Then we have $|pq_2| \ge |\tilde p \tilde q_2|$. 
\item[(2)] 
Let $p \in X$, $q_1 = \gamma(t_1)$ and $q_2 = \gamma(t_2)$ with $t_2 > t_1$ and $t_2 - t_1< D_\kappa$. 
Let $\tilde p, \tilde q_1, \tilde q_2 \in \mathbb M_\kappa$ with $|\tilde p \tilde q_1| = |pq_1|$, $|\tilde q_1 \tilde q_2| = t_2 - t_1$ and 
\[
- \cos \angle \tilde p \tilde q_1 \tilde q_2 = \limsup_{\tau \to 0+} \frac{|p \gamma(t_1+\tau)| - |p \gamma(t_1)|}{\tau}. 
\]
Then we have $|p q_2| \le |\tilde p \tilde q_2|$.  
\item[(3)]
For $p \in X$ and $t \in [a,b]$ with $|p \gamma(t)| < D_\kappa$, the comparison angle $\tilde \angle_\kappa (|p\gamma(t+\tau)|; \tau, |p\gamma(t)|)$ is non-increasing in $\tau$ for $0 \le \tau < D_\kappa$. 
\end{itemize}
\end{proposition}
Note that the limit superior in (2) is actually the limit due to the monotonicity (3). 

\begin{definition}[\cite{PP:qg}, \cite{Pet:qg}] \label{def:qg} \upshape
A rectifiable curve in a proper geodesic space is called a $\kappa$-{\it quasigeodesic} if it is $\kappa$-convex and is parametrized by arclength. 
\end{definition}

We recall that any quasigeodesic in an Alexandrov space has the forward and backward directions at any time (\cite{Pet:qg}).

\begin{theorem}[{\cite{PP:qg}}] \label{thm:qg characterization}
For a proper geodesic space $X$, if it satisfies that for any $x \in X$, there exist a neighborhood $U$ and $\kappa \in \mathbb R$ such that any two points in $U$ is joined by a minimal geodesic which is a $\kappa$-quasigeodesic in $X$, then $X$ is an Alexandrov space. 

Further, if $\kappa$ is taken independently on the choice of points $x$, then $X$ is an Alexandrov space of curvature $\ge \kappa$. 
\end{theorem}

\subsection{Extremal subsets} \label{subsec:ext}

Extremal subsets in Alexandrov spaces were originally defined in \cite{PP:ext}.

\begin{definition}[\cite{PP:ext}, \cite{Pet:semi}] \upshape \label{def:extremal}
Let $M$ be a connected Alexandrov space. 
A closed subset $E$ of $M$ is said to be {\it extremal} if 
it has one of the following equivalent properties. 
\begin{enumerate}
\item For any $p \in E$ and $q \in M$ with $q \neq p$, if $d_q$ has a local minimum at $p$ on $E$, then $p$ is a critical point for $d_q$, that is, $d_q' \le 0$ on $T_p M$.
\item For any $p \in E$ and any semiconcave function $f$ defined near $p$, its gradient curve starting at $p$ is contained in $E$. 
Here, when $p \in E \cap \partial M$, $f$ was assumed that the canonical extension of $f$ on a neighborhood of $p$ in the double of $M$ is also semiconcave. 
\item For any $p \in E$ and $q \in M$ with $q \neq p$, the gradient curve for $d_q$ starting at $p$ is contained in $E$. 
\end{enumerate}
For $\kappa > 0$ and an extremal subset $E$ of $M$, $E$ is said to be {\it $\kappa$-extremal} in $M$ if it satisfies one of the following. 
\begin{itemize}
\item $E$ consists of at least two points; 
\item $E$ consists of only one point, say $x$ and $B(x, D_\kappa/2) = M$; or
\item $E$ is the empty-set and $\mathrm{diam}\, M \le D_\kappa /2$.
\end{itemize}
When $\kappa \le 0$, any extremal subset is said to be $\kappa$-extremal. 

For a disconnected Alexandrov space $N$ and its closed subset $F$, we say that $F$ is ($\kappa$-)extremal if for each component $N_0$ of $N$, the intersection $N_0 \cap F$ is ($\kappa$-)extremal in $N_0$. 
\end{definition}

Actually, the first three conditions in the definition are equivalent.
See \cite{Pet:semi} for a proof. 
From (3), we know that if a geodesic $\gamma$ meets an extremal subset $E$ at an interior point of $\gamma$, then $\gamma$ must be contained in $E$. 
Further, it is known that each extremal subset satisfies the condition ($\spadesuit$) considered in \S \ref{subsec:length} (\cite{PP:ext}). 
It seems that the condition (2) is complicated, but it is needed when $p$ is a boundary point. 
Indeed, if $p$ is in the boundary of the unit disk $D^2= \{x \in \mathbb R^2 \mid |x| \le 1 \}$ with flat metric, then the function $f := - d(-p, \cdot)^2$ is concave on $D^2$ in our sense. 
However, its canonical extension to the double is not semiconcave and the gradient curve of $f$ starting at $p$ in $D^2$ is not contained in $\partial D^2$. 

The whole space itself and the empty subset are trivially extremal subsets.
The boundary $\partial M$ of an Alexandrov space $M$ is an extremal subset. 

For extremal subsets $E,F \subset M$, $E \cup F$, $E \cap F$ and the closure of $E \setminus F$ are also extremal.
Further, $\bigcap_\alpha E_\alpha$ is extremal if $\{E_\alpha\}$ is a family of extremal subsets.
Each extremal subset admits a canonical stratification into topological manifolds (\cite{PP:ext}). 
The {\it dimension} of an extremal subset $E$ is the maximal dimension of topological manifolds (stratum) embedded in $E$, which coincides with the topological dimension of $E$.

For a closed subset $F$ of $M$ and $p \in F$, we denote by $\Sigma_p F$ the set of all directions which is the limit of geodesic directions $\uparrow_p^{x_i}$ such that $F \setminus \{p\} \ni x_i \to p$.
From the definition, $\Sigma_p F$ is a closed subset of $\Sigma_p$. 
For a point $p$ in an extremal subset $E$ of $M$, $\Sigma_p E$ is a $1$-extremal subset of $\Sigma_p M$.
\begin{remark} \upshape \label{rem:1-extremal}
The interval $[0,\pi]$ of length $\pi$ is an Alexandrov space of curvature $\ge 1$. 
The set $\{0\}$ is extremal in $[0,\pi]$, however it is not $1$-extremal. 
\end{remark}

\begin{remark} \upshape \label{rem:dim 1 case}
Theorem \ref{thm:self gluing} holds when an Alexandrov space has dimension one, because such an Alexandrov space is actually a complete one-manifold possibly with boundary. 
This is a first step of the induction on the dimension of Alexandrov spaces in the proof of Theorem \ref{thm:self gluing}. 
\end{remark}

\begin{remark} \upshape \label{rem:dim 1 case, a}
As mentioned in the introduction, the assumption in Theorem \ref{thm:self gluing} that $M$ has at most two components is needed when $\dim M = 1$ and $M$ has positive curvature, to keep positivity of curvature of the glued space $M_f$. 
Indeed, an interval $I$ of length $\pi$ is an Alexandrov space of curvature $\ge 1$. 
Let $I_i$ denote isometric copies of $I$ for $i=1,2,3$. 
The disjoint union $X = \bigsqcup_{i=1}^3 I_i$ is also an Alexandrov space of curvature $\ge 1$, and its boundary $\partial X = \bigsqcup_{i=1}^3 \partial I_i$ consisting of six points is $1$-extremal in $X$. 
Then, an isometric involution $f$ on $\partial X$ gives the glued space $X_f$ which is a circle of length $3 \pi$. 
So, $X_f$ does not have curvature $\ge 1$. 
\end{remark}

We recall several important properties of extremal subsets. 
\begin{theorem}[Generalized Liberman's lemma \cite{PP:ext}, \cite{Pet:appl}, \cite{Pet:semi}] \label{thm:Liberman}
Let $E$ be an extremal subset of an Alexandrov space $M$ of curvature $\ge \kappa$. 
If $\gamma$ is a minimal geodesic in $E$ with respect to the length metric induced from $M$, then $\gamma$ is a $\kappa$-quasigeodesic in $M$.
\end{theorem}

\begin{theorem}[\cite{Pet:appl}] \label{thm:stability}
If Alexandrov spaces $M_i$ converge to $M$ without collapse, and if extremal subsets $E_i \subset M_i$ converges to $E \subset M$ as subsets, then $E_i$ converges to $E$ with respect to the length metrics.
\end{theorem}

\begin{lemma} \label{lem:ext isom}
Let $E$ be an extremal subset of an Alexandrov space $M$ and $p \in E$. 
For any $\epsilon > 0$, there is $\delta > 0$ such that for any $q \in B(p,\delta) \cap E$, we have $|pq|/|pq|_E < 1 +\epsilon$. 
\end{lemma}
\begin{proof}
Suppose the contrary, that is, there exist $\epsilon > 0$ and a sequence $q_i \in E$ with $q_i \to p$ such that $|pq_i|/|pq_i|_E \ge 1 + \epsilon$. 
Let $\gamma_i$ be a minimal geodesic between $p$ and $q_i$ with respect to the induced length metric on $E$. 
We take the tangent cone $T_p M$ as the limit of a scaling sequence $(\frac{1}{|pq_i|_E}M, p)$. 
Then, a sequence of curves $\gamma_i$ converges to a curve $\gamma$ contained in $T_p E$. 
Let $q_\infty \in T_p E$ denote the limit of $q_i$. 
Since, in general, a non-collapsing convergence preserves quasigeodesics (\cite{Pet:qg}, \cite{Pet:semi}), the curve $\gamma$ is a $0$-quasigeodesic in $T_p M$. 
Further, in general, for two curves $\alpha$ and $\beta$ in an Alexandrov space with $\alpha(0)=\beta(0)$, if $\alpha$ is a minimal geodesic and $\beta$ is a quasigeodesic satisfying $\alpha^+(0)=\beta^+(0)$, then by the $\kappa$-convexity of $\beta$ for some $\kappa$, we have $\alpha = \beta$ on the common domain.  
Therefore, in the cone $T_p M$, $\gamma$ must be contained in a geodesic ray starting $o$. 
%%% .....
This is a contradiction to the assumption $|q_\infty|/L(\gamma) \ge 1+\epsilon$. 
\end{proof}

\begin{remark} \upshape
As mentioned above, extremal subsets satisfy the condition ($\spadesuit$). 
It seems that Lemma \ref{lem:ext isom} states that the original metric and the length metric on extremal subsets are locally almost isometric. 
However, this is not true. 
Indeed, if $p$ is a corner point of a rectangle $X$, then points $q, r \in \partial X$ can be taken to be arbitrary close to $p$ such that $|qr|/|qr|_{\partial X} = \sqrt 2 /2$. 
\end{remark}

The following Lemma \ref{lem:dir approx} was proved in the case $E = \partial M$ in \cite{Pet:appl}, however its proof works for general extremal subsets. 
We give a proof of it for the completeness. 

\begin{lemma}[cf.\! {\cite[Lemma 2.2]{Pet:appl}}] \label{lem:dir approx}
Let $E$ be an extremal subset in an Alexandrov space $M$. 
For $p \in E$, $\eta \in \Sigma_p E$ and $\epsilon > 0$, there exists an shortest path $\alpha$ starting at $p = \alpha(0)$ in $E$ with respect to the length metric such that $\angle (\eta, \alpha^+) < \epsilon$. 

Further, for any $\epsilon > 0$, there is $\delta > 0$ such that 
for any $\eta \in \Sigma_p E$, if $q \in E \cap B(p,\delta) \setminus \{p\}$ satisfies $\angle(\uparrow_p^q, \eta) < \delta$, 
then we have $\angle(\alpha^+, \eta) < \epsilon$, where $\alpha$ is a minimal geodesic in $(E, |\cdot,\cdot|_E)$ from $p$ to $q$. 
\end{lemma}
\begin{proof}
Let $\eta \in \Sigma_p E$. 
We take a sequence $q_i \in E$ such that $q_i \to p$ and $\uparrow_p^{q_i} \to \eta$ in $\Sigma_p$ as $i \to \infty$. 
Let us fix a minimal geodesic $\gamma_i$ from $p$ to $q_i$ in $(E, |\cdot,\cdot|_E)$. 
We recall that $\gamma_i$ is a quasigeodesic in $M$ due to Theorem \ref{thm:Liberman}. 
Hence, it has the direction $\gamma_i^+ \in \Sigma_p E$. 
We prove that $\gamma_i^+$ converges to $\eta$. 
Suppose that there is an $\epsilon > 0$ such that $\angle(\gamma_i^+, \eta) \ge \epsilon$ for all $i$. 
Passing to subsequence, we may assume that $\gamma_i^+$ converges to some direction $\xi \in \Sigma_p$. 
Let us take a point $r \in M$ such that $\angle(\uparrow_p^r, \xi) < \epsilon /6$. 
Let us take $r_i \in pr$ with $|p r_i| = L(\gamma_i)$.
Further, we have 
\begin{equation} \label{eq:00}
\lim_{i \to \infty} \frac{|pq_i|}{L(\gamma_i)} = 1.
\end{equation}
We set $t_i := |p r_i| = L(\gamma_i)$ and $s_i := |pq_i|$. 
By Proposition \ref{prop:convex curve} (2), 
we have 
\begin{equation} \label{eq:01}
\csk |r_i q_i| \le \csk^2\, t_i + \kappa\, \snk^2\, t_i \cos (\epsilon/6)
\end{equation}
for large $i$, where $\kappa$ is a lower curvature bound at $p$ with $\kappa < 0$. 
We show that the comparison angle $\tilde \varphi_i := \tilde \angle_\kappa q_i p r_i$ is small enough. 
Indeed, since 
\begin{align*}
\kappa \cos \tilde \varphi_i &= 
\frac{\csk\, |r_i q_i| - \csk\, t_i\, \csk\, s_i}{\snk\, t_i \,\snk\, s_i}, 
\end{align*}
by the inequality \eqref{eq:01} and the relation \eqref{eq:00}, letting $i \to \infty$, we have 
\[
\lim_{i \to \infty} \kappa \cos \tilde \varphi_i \le \kappa \cos (\epsilon / 6).
\]
Since $\kappa < 0$, we obtain $\tilde \varphi_i < \epsilon / 5$. 
Because $\tilde \varphi_i \to \angle (\eta, \uparrow_p^r)$, we have 
\[
\angle(\eta, \xi) \le \angle (\eta, \uparrow_p^r) + \angle(\uparrow_p^r, \xi) \le 2 \epsilon /3. 
\]
This contradicts to $\angle (\eta, \xi) \ge \epsilon$. 
This completes the proof. 
\end{proof}

\begin{lemma} \label{lem:isom emb}
Let $X$ and $Y$ be Alexandrov spaces and $E \subset X$ and $F \subset Y$ their extremal subsets. 
Let $h : E \to F$ be an isometric embedding with respect to the induced length metric. 
Then, for any $p \in E$, there is an isometric embedding 
\[
h_p' : \Sigma_p E \to \Sigma_{h(p)} F
\]
with respect to the induced length metric such that 
\[
h_p'(\gamma^+(0)) = (h \circ \gamma)^+(0)
\]
for any non-trivial minimal geodesic $\gamma$ in $(E, |\cdot,\cdot|_E)$ with $\gamma(0)=p$. 
\end{lemma}
\begin{proof}
We define a map $h_p' : \Sigma_p E \to \Sigma_{h(p)} F$ as follows. 
Let $\gamma$ be a non-trivial minimal geodesic in $(E, |\cdot,\cdot|_E)$ with $\gamma(0)=p$.
Since $h$ is an isometric embedding in the length metric, $h \circ \gamma$ is a minimal geodesic in $(F, |\cdot,\cdot|_F)$, the direction $(h \circ \gamma)^+(0)$ is determined as a point in $\Sigma_{h(p)} F$. 
Later, it is known that this gives a well-posed map $h_p' : \Sigma_p E \to \Sigma_{h(p)} F$. 
For an arbitrary direction $\eta \in \Sigma_p E$ and $\epsilon > 0$, by Lemma \ref{lem:dir approx}, there is a minimal geodeisc $\alpha_\epsilon$ in $(E, |\cdot,\cdot|_E)$ with $\alpha_\epsilon(0)=p$ such that 
\begin{equation} \label{eq:approx}
\angle_p (\alpha_\epsilon^+(0), \eta) < \epsilon \text{ and } \frac{|p\alpha_\epsilon(t)|}{t} = 1 + O(t)
\end{equation}
for $t < \epsilon$, where $\angle_p$ is the distance function on $\Sigma_p$.
Let us set $h_p'(\eta) := \lim_{\epsilon \to 0} (h \circ \alpha_\epsilon)^+(0)$. 
We claim that this map is well-defined. 
Indeed, for another minimal geodeisc $\beta_\epsilon$ in $(E, |\cdot,\cdot|_E)$ with $\beta_\epsilon(0)=p$ satisfying \eqref{eq:approx}, 
by the condition ($\spadesuit$) for $\Sigma_p E$ at $\eta$, 
we have $|\alpha_\epsilon^+(0) \beta_\epsilon^+(0)|_{\Sigma_p E} < C \epsilon$ for some $C > 0$. 
Further, by Theorem \ref{thm:stability}, we obtain 
\[
\tilde \angle (|\alpha_\epsilon(t) \beta_\epsilon(t)|_E; t, t) < C \epsilon + O(t). 
\]
Since $h$ is the isometric embedding in the length metric, 
\[
\tilde \angle (|h \circ \alpha_\epsilon(t) h \circ \beta_\epsilon(t)|_F; t, t) < C \epsilon + O(t)
\]
holds. 
By using Lemma \ref{lem:ext isom} again, we have 
\[
\tilde \angle (|h \circ \alpha_\epsilon(t) h \circ \beta_\epsilon(t)|; |p \alpha_\epsilon(t)|, |p \beta_\epsilon(t)|) < C \epsilon + O(t).
\]
Therefore, $\angle_{h(p)} ((h \circ \alpha_\epsilon)^+(0), (h \circ \beta_\epsilon)^+(0)) \le C \epsilon$.
This implies that the map $h_p' : \Sigma_p E \to \Sigma_{h(p)} F$ is well-defined.

Let us give another representation of the map $h_p'$ as follows.
Since $h : (E, |\cdot,\cdot|_E) \to (F, |\cdot,\cdot|_F)$ is the isometric embedding, 
by Theorem \ref{thm:stability}, the limit map $\hat h : (T_p E, |\cdot,\cdot|_{T_p E}) \to (T_{h(p)} F, |\cdot,\cdot|_{T_{h(p)}F})$ can be defined, which is an isometric embedding.
From the above argument, the map $\hat h$ is deteremined independently on the choice of sequences in the approximation $(E, r|\cdot,\cdot|,p) \to (T_p E, |\cdot,\cdot|_{T_p E})$ as $r \to \infty$. 
By the definition, the map $h_p'$ is nothing but the map $\hat h$ restricted to $\Sigma_p E$. 
This completes the proof.
\end{proof}

\section{Proof of Theorem \ref{thm:self gluing}} \label{sec:proof}

Let us first give a proof of Corollary \ref{cor:partial gluing}. 
\begin{proof}[Proof of Corollary \ref{cor:partial gluing}]
Let $M_1$, $M_2$, $E_1$, $E_2$ and $f$ be as in Corollary \ref{cor:partial gluing}. 
Then, the disjoint union $E := E_1 \sqcup E_2$ is $\kappa$-extremal in the Alexandrov space $M := M_1 \sqcup M_2$. 
The map $f : E_1 \to E_2$ implies an isometric involution $\bar f : E \to E$ with respect to the length metric. 
Since $M_1 \cup_f M_2$ is nothing but $M_{\bar f}$, by Theorem \ref{thm:self gluing}, we obtain the conclusion. 
\end{proof}

Let us prove Theorem \ref{thm:self gluing}. 
\begin{proof}[Proof of Theorem \ref{thm:self gluing}]
Let $M$ be an $n$-dimensional Alexandrov space of curvature $\ge \kappa$ and $E$ its $\kappa$-extremal subset of dimension $n-1$. 
Let $f : E \to E$ be a non-trivial isometric involution with respect to the length metric.
By Remark \ref{rem:dim 1 case}, we may assume that $n \ge 2$. 
By Lemma \ref{lem:length}, any point in $M_f \setminus E_f$ has a neighborhood of curvature $\ge \kappa$. 
Therefore, due to the globalization theorem (\cite{BGP}), it suffices to prove that any point in $E_f$ has a neighborhood of curvature $\ge \kappa$. 

Let us take $x \in E$. 
By Lemma \ref{lem:isom emb}, we obtain an isometry 
\begin{equation} \label{eq:derivation}
f_x' : \Sigma_x E \to \Sigma_{f(x)} E
\end{equation}
with respect to the length metric. 
Further, if $x$ is a fixed point of $f$, then the map $f_x'$ is an involution. 

Since $\Sigma_x E$ and $\Sigma_{f(x)}E$ are 1-extremal in $\Sigma_x M$ and $\Sigma_{f(x)} M$ of codimension one, the inductive hypothesis implies that the metric space 
\[
\Sigma_x^\sharp := \left\{ 
\begin{aligned}
&(\Sigma_x)_{f_x'} &&\text{if } x \text{ is a fixed point of } f, \\
&\Sigma_x \cup_{f_x'} \Sigma_{f(x)} &&\text{otherwise}
\end{aligned}
\right.
\]
is an Alexandrov space of curvature $\ge 1$. 
Let $\Sigma_x E_f \subset \Sigma_x^\sharp$ be the subset corresponding to $\Sigma_x E \subset \Sigma_x$ and $\Sigma_{f(x)} E \subset \Sigma_{f(x)}$. 

The $\kappa$-cone extension $D f_x$ of $f_x'$ gives an isometry 
\[
D f_x : T_x^\kappa E \to T_{f(x)}^\kappa E 
\]
between the $\kappa$-cones. 
Note that the $\kappa$-cone over $\Sigma_x^\sharp$ is isometric to the glued space 
\[
T_x^\sharp := 
\left\{ 
\begin{aligned}
&(T_x^\kappa M)_{D f_x} && \text{if } x = f(x), \\
&T_x^\kappa M \cup_{D f_x} T_{f(x)}^\kappa M &&\text{if } x \neq f(x). 
\end{aligned}
\right.
\]
Hence, $T_x^\sharp$ is an Alexandrov space of curvature $\ge \kappa$.

\begin{lemma}[{\cite[Lemma 2.5]{Pet:appl}}] \label{lem:tangent cone}
For any sequence $r_i \to \infty$, there is a subsequence $r_{i_j}$ of it and a pointed proper geodesic space $(K,o)$ such that 
$(r_{i_j} M_f, x)$ converges to $(K,o)$ as $j \to \infty$, in the Gromov-Hausdorff sense. 

Further, for any such a limit $K$, there is a distance non-contracting map $\ell : K \to T_{o_x} T_x^\sharp$ such that $|\ell(v)| = |o v|$ for any $v \in K$, where 
$T_{o_x} T_x^\sharp$ is the tangent cone of the cone $T_x^\sharp$ at the origin $o_x$. 
\end{lemma}
\begin{proof}
Recall that the $\kappa$-gradient-exponential maps 
\[
\gexp_x^\kappa : T_x^\kappa M \to M \text{ and } \gexp_{f(x)}^\kappa : T_{f(x)}^\kappa M \to M
\]
are defined, respectively.
Further, for any $v \in T_x^\kappa E$, we have 
\[
f(\gexp_x^\kappa (v)) = \gexp_{f(x)}^\kappa(D f_x (v)).
\]
Therefore, the map $\gexp^\sharp : T_x^\sharp \to M_f$ can be defined as 
\[
\gexp^\sharp(v) := \left\{
\begin{aligned}
&\gexp_x^\kappa(v) &&\text{ if } v \in T_x^\kappa M \\
&\gexp_{f(x)}^\kappa(v) &&\text{ if } v \in T_{f(x)}^\kappa M. 
\end{aligned}
\right.
\]
Since $\gexp_x^\kappa$ and $\gexp_{f(x)}^\kappa$ are $1$-Lipschitz, so is 
$\gexp^\sharp$ due to Lemma \ref{lem:predist}. %%%....
From the construction, we have 
\begin{equation} \label{eq:norm preserving}
\left. \frac{d}{d t} \right|_{t = 0+} |\gexp^\sharp(t v) x| = |v|.
\end{equation}
For a rigorous proof of \eqref{eq:norm preserving}, we refer the proof of \cite[Lemma 2.5]{Pet:appl}. 
Further, the image of $\gexp^\sharp$ contains $B(x,D_\kappa /2)$ which is a subset of $M_f$. 
Therefore, the first statement holds. 
Let us consider any limit $K$ of scaling spaces $(r_i T_x^\sharp, o_x)$ and $(r_i M_f, x)$ as $r_i \to \infty$, and the limit $D(\gexp^\sharp)$ of $\gexp^\sharp$ under this convergence. 
Since $T_x^\sharp$ is an Alexandrov space, its scaling limit is no other than the unique tangent cone $T_{o_x} T_x^\sharp$. 
So, we have a $1$-Lipschitz surjective map 
\[
D (\gexp^\sharp) : T_{o_x} T_x^\sharp \to K. 
\]
Here, $K$ 
denotes the limit of $(r_i M_f, x)$. 
By \eqref{eq:norm preserving}, it satisfies
\[
|D(\gexp^\sharp)(v)| = |v|
\]
for all $v \in T_{o_x} T_x^\sharp$.
Since $D(\gexp^\sharp)$ is surjective, its left-inverse 
\[
\ell : K \to T_{o_x} T_x^\sharp
\]
gives the desired map. 
This completes the proof. 
\end{proof}

Let us take $p, q \in M_f$ with $|pq| < \infty$ and $x_0 x_1 \dots x_k x_{k+1}$ an $m$-shortest path between $x_0 = p$ and $x_{k+1} = q$, where $k \le m$. 
Let us fix $x = x_i \in \{x_1, \dots, x_k\} \subset E_f$ an inner vertex of the $m$-shortest path.

\begin{lemma}[{\cite[Lemma 2.6]{Pet:appl}}] \label{lem:antipodal}
The directions $\xi :=\, \uparrow_{x_i}^{x_{i-1}}$ and $\eta :=\, \uparrow_{x_i}^{x_{i+1}}$ are antipodal in $\Sigma_x^\sharp$.
Namely, we have
\[
|\xi \zeta| + |\eta \zeta| = \pi
\]
for all $\zeta \in \Sigma_x^\sharp$.
Here, $|\cdot,\cdot|$ is the canonical length metric on the glued space $\Sigma_x^\sharp$.
\end{lemma}
\begin{proof} 
Let $\xi \in \Sigma_x M$ and $\eta \in \Sigma_{f(x)} M$.
Let us first prove that 
\begin{equation} \label{eq:antipodal}
|\xi \nu|_0 + |\eta \nu|_0 = \pi
\end{equation}
for all $\nu \in \Sigma_x E_f$.
Let us set $g(\nu) := |\xi \nu|_0 + |\eta \nu|_0$. 
Since $x_{i-1} x x_{i+1}$ is a $1$-shortest path, 
applying the usual first variation formula, 
we obtain 
\[
- \cos |\xi \nu|_0 - \cos |\eta \nu|_0 \ge 0
\]
for all $\nu \in \Sigma_x E_f$.
Hence, we have 
\[
g(\nu) = |\xi \nu|_0 + |\eta \nu|_0 \ge \pi.
\]
Let us take $\bar \nu \in \Sigma_x E_f$ a minimizer of the function $g$. 
Then, $\xi \bar \nu \eta$ is a $1$-shortest path in $\Sigma_x^\sharp = \Sigma_x M \cup_{f_x'} \Sigma_{f(x)} M$.
Let us take a minimal geodesic $c$ in $\Sigma_x E_f$ with respect to the length metric starting from $c(0) = \bar \nu$ in the direction $c^+(0) \in \Sigma_{\bar \nu} \Sigma_x E_f$.
By the inductive hypothesis about the property \eqref{eq:antipodal}, we have 
\[
|\! \uparrow_{\bar \nu}^\xi c^+(0)|_0 + |\! \uparrow_{\bar \nu}^\eta c^+(0)|_0 = \pi.
\]
By Theorem \ref{thm:Liberman}, the curve $c$ is a $1$-quasigeodesic both in $\Sigma_x M$ and $\Sigma_{f(x)} M$.
If $g(\bar \nu) > \pi$, then by the hinge comparison (Proposition \ref{prop:convex curve}(2)) for the hinges $(c, \bar \nu \xi)$ and $(c, \bar \nu\eta)$, we have 
\[
g(c(t)) < g(c(0)) = g(\bar \nu)
\]
for small $t > 0$. 
This is contradict to the choice of $\bar \nu$.
Hence, we obtain $g(\bar \nu) = \pi$.

Let us define a function $h : \Sigma_x E_f \to \mathbb R$ by 
\[
h(\nu) := - \cos |\xi \nu|_0 - \cos |\eta \nu|_0. 
\]
Let $c$ denote a minimal geodesic in $\Sigma_x E_f$ with respect to the length metric starting at $c(0) = \bar \nu$. 
Since $\bar \nu$ minimize $h$, we have $h \circ c \ge h(\bar \nu) = 0$.
By Theorem \ref{thm:Liberman}, $h \circ c$ is spherically concave. 
Further, we have 
\begin{align*}
(h \circ c)'(0) = \cos |\xi \bar \nu|_0 \cos |\!\uparrow_{\bar \nu}^\xi c^+(0)|_0 + 
\cos |\eta \bar \nu|_0 \cos |\!\uparrow_{\bar \nu}^\eta c^+(0)|_0 = 0.
\end{align*}
By the spherically concavity of $h \circ c$, we obtain 
\[
h \circ c(t) \le h(\bar \nu) \cos t + (h \circ c)'(0) \sin t = 0.
\]
Therefore, $h = 0$ on $\Sigma_x E_f$.
Namely, \eqref{eq:antipodal} is proved.

We prove $|\xi \eta| = \pi$ in $\Sigma_x^\sharp$.
Suppose $|\xi \eta| <\pi$. 
Then, there is an $m$ such that $|\xi \eta|_m < \pi$.
Let $\theta$ be the closest vertex to $\xi$ of an $m$-shortest path between $\xi$ and $\eta$.
Since $\theta \in \Sigma_x E_f$, $\xi \theta \eta$ is a $1$-shortest path of length $\pi$ in $\Sigma_x^\sharp$ by \eqref{eq:antipodal}. 
Hence, there are two distinct directions at $\theta$ which are opposite to $\uparrow_\theta^\xi$ in $\Sigma_\theta \Sigma_x^\sharp$.
It is contradict to that $\Sigma_x^\sharp$ is an Alexandrov space. 
This completes the proof of Lemma \ref{lem:antipodal}.
\end{proof} 

\begin{corollary}[{\cite[Corollary 2.7]{Pet:appl}}] \label{cor:antipodal}
Let $\xi$ and $\eta$ be an in Lemma \ref{lem:antipodal}. 
For $\zeta \in \Sigma_x^\sharp$, there is a unique $\zeta^\ast \in \Sigma_x E_f$ such that 
\begin{align*}
|\xi \zeta|_0 + |\zeta \zeta^\ast|_0 + |\zeta^\ast \eta|_0 = \pi 
\end{align*}
or
\begin{align*}
|\eta \zeta|_0 + |\zeta \zeta^\ast|_0 + |\zeta^\ast \xi|_0 = \pi.
\end{align*}
\end{corollary}
\begin{proof}
We may assume $\xi \in \Sigma_x M$, $\eta \in \Sigma_{f(x)} M$ and $\zeta \in \Sigma_x M$.
Let us consider a $1$-shortest path $\zeta \zeta^\ast \eta$ between $\zeta$ and $\eta$ in $\Sigma_x^\sharp$ through some $\zeta^\ast \in \Sigma_x E_f$. 
%%%....
Then, such a $\zeta^\ast$ is the desired direction. 
\end{proof}

\begin{lemma}[{\cite[Lemma 2.8]{Pet:appl}}] \label{lem:m-dist comp}
Let $\gamma : [a,b] \to M_f$ be a minimal geodesic in $E_f$ with respect to the length metric.
Then 
\[
\frac{d^2}{d t^2}\rho_\kappa(|p \gamma(t)|_m) + \kappa \rho_\kappa(|p \gamma(t)|_m) \le 1
\]
on a neighborhood of each $t_0$, 
for any $p \in M_f \setminus \{\gamma(t_0)\}$ with $|p \gamma(t_0)|_m < D_\kappa /4$ and $m \ge 0$.
\end{lemma}
Petrunin proved the corresponding lemma for the case $\kappa = 0$ in the proof of Theorem \ref{thm:Pet gluing}. 
In this paper, we prove the lemma when $\kappa \neq 0$. 
The statement of the case $\kappa =0$ also follows from the statements of the case $\kappa < 0$ and a limiting procedure.
\begin{proof}[Proof of Lemma \ref{lem:m-dist comp}]
Let us assume that $\kappa \neq 0$.
Let us set $h(u) := h_{\kappa}(u) := - \kappa^{-1} \csk\, u$. 
To prove Lemma \ref{lem:m-dist comp}, it suffices to prove that 
\begin{equation} \label{eq:m-dist comp}
\frac{d^2}{d t^2} h (|p\gamma(t)|_m) + \kappa h (|p \gamma(t)|_m) \le 0 
\end{equation}
for $m \ge 0$, $\gamma$ as in the assumption and $t$ with $|p \gamma(t)|_m < D_\kappa /4$.

Since the 0-predistance $|p\gamma(t)|_0$ is no other than the original distance from the set $p \subset M$ to the curve $\gamma$, 
by Theorem \ref{thm:Liberman}, the lemma is true for $m=0$.

We suppose that the lemma is true for all $\ell < m$ and false for $m$.
Namely, there exist a curve $\gamma$ as in the assumption in Lemma \ref{lem:m-dist comp} and a point $p$, 
the function $h(|p \gamma(t)|_m)$ does not satisfy \eqref{eq:m-dist comp}.
From the inductive hypothesis, the function $h(|p \gamma(t)|_m)$ satisfies the assumption \eqref{eq:a1} in Lemma \ref{lem:contradiction1}. 
Indeed, this is valid by a later equation \eqref{eq:later}. 
Hence, there exist a parameter $t_0$ at an interior point of $\gamma$, $\epsilon > 0$ and $A \in \mathbb R$ such that for any $t$ with $|t-t_0| < \epsilon$, 
\begin{align} \label{eq:contradiction}
h(|p \gamma(t)|_m) \ge \hspace{0.2em}& h(|p \gamma(t_0)|_m) -\kappa h(|p \gamma(t_0)|_m) (t-t_0)^2/2 \\ 
&+ \epsilon (t-t_0)^2 + A (t-t_0). \nonumber 
\end{align}
We may assume $t_0 = 0$ by translating the parameter of $\gamma$.

Let $q := \gamma(0)$ and let us fix an $m$-shortest path $p_0 p_1 \dots p_k p_{k+1}$ between $p = p_0$ and $q = p_{k+1}$, where $k \le m$. 
We may assume $k \ge 1$.
Let us take a sequence $t_j$ with $\lim_{j \to \infty} t_j = 0$. 
Let $(\uparrow_{p_k}^{\gamma(t_j)})^\ast \in \Sigma_{p_k} E_f \subset \Sigma_{p_k}^\sharp$ be a direction taken by Corollary \ref{cor:antipodal}.
Taking a subsequence of $\{t_j\}$, we may assume that  $\{(\uparrow_{p_k}^{\gamma(t_j)})^\ast\}$ converges to some direction $\nu \in \Sigma_{p_k} E_f$.
For any $\delta > 0$, by Lemma \ref{lem:dir approx}, there is a minimal geodesic $\sigma$ in $N$ from $p_k$ such that 
\[
\angle(\sigma^+(0), \nu) < \delta.
\]
Here, $\angle$ is the distance function on $\Sigma_{p_k}^\sharp$.
We set 
\begin{align*}
&\alpha = \angle (\uparrow_q^{p_k}, \gamma^+(0)), 
&&\beta = \angle(\uparrow_{p_k}^q, \sigma^+(0)), \\
&\beta_j = \angle(\uparrow_{p_k}^{\gamma(t_j)}, \sigma^+(0)), 
&&\theta_j = \angle(\uparrow_{p_k}^{\gamma(t_j)}, \uparrow_{p_k}^q).  
\end{align*}
Here, we may assume that $\uparrow_{p_k}^{\gamma(t_j)} \to \uparrow_{p_k}^q$ as $j \to \infty$. 
Hence, $\theta_j \to 0$ and $\beta_j \to \beta$ as $j \to \infty$.
Note that by Corollary \ref{cor:antipodal}, $\uparrow_{p_k}^{\gamma(t_j)}$ is contained in a geodesic between $\uparrow_{p_k}^q$ and $(\uparrow_{p_k}^{\gamma(t_j)})^\ast$ in $\Sigma_{p_k}^\sharp$.

From the monotonicity of the comparison angle (Proposition \ref{prop:convex curve}(3)) and the law of sine in the $\kappa$-plane, we have 
\[
\frac{\snk\, \theta_j}{\snk\, |t_j|} 
\ge 
\frac{\snk\, \tilde \angle q p_k \gamma(t_j)}{\snk\, |t_j|} = \frac{\snk\, \tilde \angle p_k q \gamma(t_j)}{\snk\, |p_k \gamma(t_j)|_0} 
= \frac{\snk\, \alpha}{\snk\, |p_kq|_0} + O(t_j). 
\]

We can assume that $\uparrow_{p_k}^q \not\in \Sigma_{p_k} E_f$. %%%...
By taking into account the triangle $\triangle \uparrow_{p_k}^q \uparrow_{p_k}^{\gamma(t_j)} \sigma^+(0)$ in $\Sigma_{p_k}^\sharp$, 
we have 
\begin{align*}
\frac{\cos \beta \cos \theta_j - \cos \beta_j}{\sin \beta \sin \theta_j}
&\le - \cos \angle \uparrow_{p_k}^{\gamma(t_j)} \uparrow_{p_k}^q \sigma^+(0) \\
&= - \cos \angle \left(\uparrow_{p_k}^{\gamma(t_j)}\right)^\ast \uparrow_{p_k}^q \sigma^+(0)\\
&\le - \cos \left( O(\delta) + O(t_j) \right)\!.
\end{align*}
Therefore, we have 
\[
\cos (\beta - \beta_j) 
\le 1 - \frac{t_j^2 \sin^2 \alpha}{2\, \snk^2\, |p_k q|_0} + O(\delta) t_j^2 + o(t_j^2).
\]

By the inductive hypothesis, we have 
\[
(h(|p \sigma(\tau)|_{m-1}))'' + \kappa h(|p \sigma(\tau)|_{m-1}) \le 0.
\]
From Lemma \ref{lem:predist}, we have
\[
h(|p \sigma(\tau)|_{m-1}) \le h(|p p_k|_{m-1})\, \csk\, \tau + \snk\, |p p_k|_{m-1}\, \snk\, \tau \cos \beta. 
\]
Since $\sigma$ is a quasigeodesic in $M$, we obtain 
\[
h(|\gamma(t_j) \sigma(\tau)|_0) \le 
h( |\gamma(t_j) p_k|_0)\, \csk\, \tau - \snk\, |\gamma(t_j) p_k|_0\, \snk\, \tau \cos \beta_j. 
\]

Let us see the figure in \cite[p.217]{Pet:appl} and regard it as a figure in the $\kappa$-plane. 
Note that when $j$ is large, the length of $\sigma$ is sufficiently greater than $|\gamma(t_j) p_k|_0$. 
Therefore, we obtain the second inequality in the following.
The first inequality comes from the monotonicity of $h$.  
\begin{align} 
h(|p \gamma(t_j)|_m) 
&\le h\left(\min_{\tau} \left(|p \sigma(\tau)|_{m-1} + |\sigma(\tau) \gamma(t_j)|_0 \right) \right) \nonumber \\
&\le \csk\, |p p_k|_{m-1}\, h(|p_k \gamma(t_j)|_0) \label{eq:later} \\
&\hspace{20pt} + \snk\, |p p_k|_{m-1}\, \snk\, |p_k \gamma(t_j)|_0 \cos (\beta - \beta_j). \nonumber
\end{align}
Note that, as mentioned above, this implies that the function $h(|p \gamma(t)|_m)$ satisfies \eqref{eq:a1} in Lemma \ref{lem:contradiction1}. 

By Theorem \ref{thm:Liberman}, $\gamma$ is a quasigeodesic in $M$. 
Hence, we have 
\[
h(|p_k \gamma(t_j)|_0) \le h(|p_k q|_0)\, \csk\, t_j + \snk\, |p_k q|_0\, \snk\, t_j \cos \alpha.
\]
By the first variation formula, we obtain 
\[
\snk\, |p_k \gamma(t_j)|_0 \le \snk\, |p_k q|_0 - t_j \,\csk\, |p_k q|_0\, \cos \alpha + \frac{\lambda}{2} t_j^2 + o(t_j^2)
\]
where $\lambda$ is given as 
\[
\lambda = \frac{\csk^2\, |p_k q|_0 - \cos^2 \alpha}{\snk\, |p_k q|_0}
\]
due to Lemma \ref{lem:sine comp}. 

By summarizing the above inequalities and using the expansions $\csk\, T = 1 - (\kappa/2) T^2 + o(T^2)$ and $\snk\, T = T + o(T^2)$, we obtain 
\[
\epsilon t_j^2 + A' t_j \le O(\delta) t_j^2 + o(t_j^2)
\]
for some constant $A' \in \mathbb R$. 
Tending $t_j \to \pm 0$, we have $A' = 0$ and $\epsilon \le |O(\delta)|$. 
The last inequality fails when $\delta$ is small. 
This completes the proof of Lemma \ref{lem:m-dist comp}.
\end{proof} 
Let us continue the proof of Theorem \ref{thm:self gluing}. 
Let $\gamma = p_0 p_1 \dots p_{k+1}$ be an $m$-shortest path in $M_f$. 
From Lemma \ref{lem:m-dist comp}, we know that $\gamma$ is $\kappa$-convex with respect to the $\ell$-predistance at any point except interior points $p_i = \gamma(t_i)$, $1 \le i \le k$. 
Passing $\ell \to \infty$, we have that $\gamma$ is $\kappa$-convex with respect to the length metric $|\cdot, \cdot|$ on $M_f$ at any $t$ except $t_i$. 
Let us fix $i$ and consider $p_i = \gamma(t_i)$ an interior point. 
Let us take $x \neq p_i$ and a geodesic $\sigma$ starting at $p_i$ ending at $x$. 
By Lemmas \ref{lem:tangent cone} and \ref{lem:antipodal}, for each fixed $\epsilon > 0$, we have 
\[
|\sigma(T) \gamma(t_i + T\epsilon)| + |\sigma(T) \gamma(t_i-T \epsilon)| \le 2 T + 10 T \epsilon^2 + o(T).
\]
Therefore, we obtain 
\begin{align*}
|x \gamma(t_i + T\epsilon)| + |x \gamma(t_i - T\epsilon)| 
&\le 2 |x \sigma(T)| + 2 T + 10 T \epsilon^2 + o(T) \\
&= 2 |x p_i| + 10 T \epsilon^2 + o(T).
\end{align*}
Tending $T \to 0$ and $\epsilon \to 0$, we have 
\[
|x \gamma(t)|^+(t_i) - |x \gamma(t)|^-(t_i) \le 0.
\]
Therefore, from \cite[1.3 (2)]{PP:qg}, $\gamma$ is known to be $\kappa$-convex at any time. 

Let $\gamma_m$ be an $m$-shortest path between $x$ and $y$ in $X$. 
Then, the limit $\gamma = \lim_{m \to \infty} \gamma_m$ is a minimal geodesic between $x$ and $y$. 
Then, $\gamma$ is $\kappa$-convex and parametrized by arclength, because so is each $\gamma_m$ and $L(\gamma_m) \to L(\gamma)$. 
Hence, $\gamma$ is a $\kappa$-quasigeodesic. 
By Theorem \ref{thm:qg characterization}, we know that $X$ is an Alexandrov space of curvature $\ge \kappa$. 
This completes the proof of Theorem \ref{thm:self gluing}. 
\end{proof} 

\begin{example} \upshape \label{ex:collapse}
We see that the gluing construction as in Corollary \ref{cor:partial gluing} naturally appear when considering collapsing Alexandrov spaces with boundaries. 
Let us fix an interval $I= [0,1]$. 
Then, an Alexandrov surface $M$ which is Gromov-Hausdroff close to $I$ is topologically classified. 
When $M$ has non-empty boundary, $M$ is homeomorphic to a disk, an annulus or a M\"obius band. 
One can give several families of nonnegatively curved metrics on such an $M$ converging to $I$ as follows. 
We consider the case $M$ is a disk. 
Let $S^1_\epsilon$ a circle of length $\epsilon$ and let $M_\epsilon = K_1(S^1_\epsilon)$ which is the closed ball in the Euclidean cone over $S^1_\epsilon$ centered at the origin and of radius $1$. 
Then, $M_\epsilon$ collapses to $I$ as $\epsilon \to 0$. 
Further, its double $D(M_\epsilon)$ collapses to the partial double $I \cup_f I$ of $I$, where $f$ is the identity on $\{1\} \subset \partial I$. 

In \cite{MY:bdry}, collapsing three-dimensional Alexandrov spaces with boundaries were classified. 
To state results in \cite{MY:bdry}, we needed the terminology of partial-double, that is, gluing of two same Alexandrov spaces along the same extremal subsets of codimension one via the identity. 
For instance, let $N_\epsilon = K(S^1_\epsilon) \times \mathbb R_{\ge 0}$, the product of the cone over a circle of length $\epsilon$ and the half-line. 
It converges to $(\mathbb R_{\ge 0})^2$ as $\epsilon \to 0$. 
Then, the double $D(N_\epsilon)$ converges to $\mathbb R \times \mathbb R_{\ge 0}$ which is the glued space of two $(\mathbb R_{\ge 0})^2$'s along an extremal subset $\mathbb R_{\ge 0} \times \{0\}$ via the identity, that is, 
$\mathbb R \times \mathbb R_{\ge 0}$ is obtained as 
$\{ (\mathbb R_{\ge 0})^2 \sqcup (\mathbb R_{\ge 2})^2 \}_{\mathrm{id} : \mathbb R \times \{0\} \to \mathbb R \times \{0\} }$ in our terminology. 
\end{example}

Extracting sufficient conditions for proving Theorem \ref{thm:self gluing}, 
the statements of Theorem \ref{thm:self gluing} and Corollary \ref{cor:partial gluing} are generalized to the following form. 
(However, it is hard to check the assumption of Theorem \ref{thm:non-ext gluing}, for general Alexandrov spaces).
\begin{theorem} \label{thm:non-ext gluing}
Let $M$ be an Alexandrov space and $E$ its codimension one extremal subset. 
Let $f : E \to \partial M$ be an isometric embedding with respect to the length metric. 
For $x \in E$, let $f_x' : \Sigma_x E \to \Sigma_{f(x)} E$ denote an isometric embedding with respect to the length metric obtained by Lemma \ref{lem:isom emb}. 
Then, we have 
\begin{enumerate}
\item If the metric space $(\Sigma_x)_{f_x'} = \Sigma_x /\!\! \sim$ obtained by gluing $\Sigma_x E$ and its image $f_x'(\Sigma_x E)$ along $f_x'$ in $\Sigma_x$ has curvature $\ge 1$ for every $x \in E$, then the metric space $M_f = M /\!\! \sim$ obtained by gluing $E$ and $f(E)$ along $f$ in $M$ is an Alexandrov space. 
\item Further, under the assumption of (1), if $M$ has curvature $\ge \kappa$, $E$ is connected and is $\kappa$-extremal in $M$, then $M_f$ has curvature $\ge \kappa$. 
\end{enumerate}
\end{theorem}
To obtain an Alexandrov space by gluing, the assumption of (1) is needed. 
Actually, by gluing two rectangles of different sizes, we obtain a metric space admitting branching geodesics. 
%%, shown in the following picture. 
%%\[
%%\begin{matrix}
%%\includegraphics[bb=0 0 360 270,width=6em, height=4em]{2rect0.eps}
%%\end{matrix}
%%= 
%%\begin{matrix}
%%\includegraphics[bb=0 0 360 270,width=5.5em, height=3.8em]{2rect1.eps}
%%\end{matrix}
%%\]

\begin{remark}\upshape 
We can not replace an isometric involution in the statement of Theorem \ref{thm:self gluing} with an isometric group action. 
Indeed, let us consider a disk with flat metric and an effective isometric action of the cyclic group of order three on the boundary in the induced length metric. 
Then, the quotient space of the disk obtained by gluing the boundary along the isometries is not an Alexandrov space. 
Indeed, at a former boundary point, the space of directions is not an Alexandrov space, because it is a graph of degree three consisting of two nodes and three edges. 
\end{remark}

\end{document}